\documentclass[11pt]{amsart}
\usepackage[utf8]{inputenc}
\usepackage[english]{babel}

\usepackage[pdftex,svgnames,dvipsnames]{xcolor}
\usepackage[a4paper,top=4cm,bottom=3cm,left=4cm,right=4cm]{geometry}
\usepackage{amsmath,amsthm,amssymb,amsfonts}
\usepackage{mathrsfs, mathtools}
\usepackage{thmtools}
\usepackage{import}
\usepackage{enumitem}
\usepackage{xifthen}
\usepackage{pdfpages}
\usepackage{transparent}
\usepackage{float}
\usepackage{comment}
\usepackage{subcaption}
\usepackage{hyperref}
\usepackage{tikz-cd}
\usepackage{stackrel}
\usepackage{wrapfig}
\usepackage[misc]{ifsym}
\usepackage[pagewise]{lineno}
\usepackage{bm}
\usepackage{cleveref}

\graphicspath{{figures/}}

\newcommand{\bbR}{\mathbb{R}}

\newcommand{\bbZ}{\mathbb{Z}}
\newcommand{\bbN}{\mathbb{N}}
\newcommand{\bbT}{\mathbb{T}}

    \newcommand{\graph}{\operatorname{graph}}

    \newcommand{\loc}{\text{loc}}
    \newcommand{\floc}{f_{\text{loc}}}

    \newcommand{\subof}{\subset}
    \newcommand{\ti}{\times}
    \newcommand{\sans}{\setminus}

\newcommand{\Es}{E^s}

\newcommand{\Eu}{E^u}
\newcommand{\Ecu}{E^{cu}}
\newcommand{\Ecs}{E^{cs}}

\newcommand{\hypW}{W_{\operatorname{hyp}}}
\newcommand{\hyp}[1]{W_{#1}}

\newcommand{\Wu}{W^u}

\newcommand{\inv}{^{-1}}

    \newcommand{\invn}{^{-n}}

    \newcommand{\del}{\partial}

    \newcommand{\cube}{R}

    \newcommand{\Bone}{\mathcal{B}}
    \newcommand{\Cone}{\mathcal{C}}
    \newcommand{\Fcal}{\mathcal{F}}
    \newcommand{\Ucal}{\mathcal{U}}

    \newcommand{\fcovers}{\xRightarrow{f}}
    \newcommand{\fweakcovers}{\xrightarrow{f}}

    \newcommand{\Lam}{\Lambda}

    \newcommand{\gam}{\gamma}
    \newcommand{\Sig}{\Sigma}
    \newcommand{\sig}{\sigma}

    \newcommand{\al}{\alpha}
    
    \newcommand{\bt}{\beta}

\newcommand\restr[2]{{ 
  \left.\kern-\nulldelimiterspace 
  #1 
  \vphantom{\big|} 
  \right|_{#2} 
  }}

\numberwithin{equation}{section}
\theoremstyle{plain}

\newtheorem{theorem}{Theorem}[section]
\newtheorem{maintheorem}{Theorem}

\newtheorem{proposition}[theorem]{Proposition}
\newtheorem{lemma}[theorem]{Lemma}
\newtheorem{corollary}[theorem]{Corollary}

\theoremstyle{definition}

\newtheorem{definition}[theorem]{Definition}
\newtheorem*{notation}{Notation}

\newtheorem{remark}{Remark}

\title[A computer-assisted proof of robust transitivity]{A computer-assisted proof of robust transitivity}


\author{Marisa Cantarino}
\address{Centro de Modelamiento Matemático (CNRS IRL2807)\\Universidad de Chile, Santiago, Chile}
\email{\Letter \; marisa.cantarino@cmm.uchile.cl}

\author{Andy Hammerlindl}
\address{School of Mathematics, Monash University\\ Clayton, VIC 3800, Australia}
\email{andy.hammerlindl@monash.edu}

\author{Warwick Tucker}
\address{School of Mathematics, Monash University\\ Clayton, VIC 3800, Australia}
\email{warwick.tucker@monash.edu}

\thanks{This research was partially supported by the Australian Research Council through grant DP220100492. M. C. was partially financed the Centro de Modelamiento Matemático (CMM) BASAL fund FB210005 for center of excellence from ANID-Chile. We thank the mathematical research institute MATRIX in Australia where the authors started helpful conversations with Maciej Capiński, Lorenzo Díaz, Bernd Krauskopf, Natalia McAlister and Hinke Osinga. The authors also thank Sylvain Crovisier and Rafael Potrie for discussing related results.}

\begin{document}

\begin{abstract}
    We present computer-assisted proofs of partial hyperbolicity, existence of a blender and robust transitivity for diffeomorphisms on closed manifolds. These proofs are implemented for a family of derived-from-Anosov systems on the 3-torus.
\end{abstract}

\maketitle


\section{Introduction}

The goal of this paper is to provide computer-assisted proofs of topological and dynamical properties of differentiable dynamical systems in manifolds that can be used to establish robust transitivity. In particular, we have pure mathematical statements implying partial hyperbolicity, existence of a transitive uniformly hyperbolic set, existence of a blender, and denseness of stable and unstable manifolds. These results are tailored to be implemented using interval arithmetic.

A computer-assisted proof is an argument with formal mathematical validity which is partially performed by a computer. Symbolic manipulation of algebraic formulas or graph properties obtained by algorithms are simple examples, but computers can also be used to check inequalities formally, in particular with the use of intervals.

For differentiable dynamical systems, there is a great interest in knowing if a property is present in any meaningful way on a topological space of $C^1$ functions. This could mean, for instance, that this property is \textit{generic}, meaning that it is present for a typical system under this topology. Or we may look for \textit{robust} properties, those that, once present for a particular system satisfying some conditions, are also present for nearby systems. The set of maps having some robust property forms an open set, so this means this property should be taken into account when looking for typical dynamics for the topology.

Robust properties are also of particular interest in applications, as we may often approximate the function under consideration with a nearby one, which can be obtained experimentally or could be a numerical representation of the original function's behavior. 

One way we can ensure robustness of topological properties is through \textit{structural stability}, which means that, given a function, all sufficiently $C^1$-close functions have the same topological dynamics, or more precisely, they are topologically conjugate. This turns out to be a quite restricted scenario, since structural stability is equivalent to hyperbolicity and a strong transversality condition, as conjectured by Palis and Smale \cite{Palis-Smale} and whose proof was finished by Ma\~n\'e \cite{Mane88}, following the contributions of other authors \cite{robbin71, melo73, robinson73, robinson76}.

One topological property related to chaos is \textit{transitivity}, which for differentiable dynamical systems on manifolds means that there are dense orbits for the system. For surfaces, the only \textit{robustly transitive} sets are the hyperbolic ones \cite{mane82}. For a few decades, the only known examples of systems presenting transitivity in a robust way were uniformly hyperbolic. Then, in the 1970s, both Shub \cite{shub71} and Ma\~n\'e \cite{mane1978} introduced examples of systems that were robustly transitive but not uniformly hyperbolic. These are examples of what we nowadays call \textit{partially hyperbolic} diffeomorphisms. Another more general type of system that is robustly transitive was introduced through \textit{blenders} \cite{bonatti-diaz}, which are associated with invariant, hyperbolic, and transitive subsets that exhibit a specific topological behavior robustly (see \Cref{def:blender}).

Conditions for robust transitivity have been explored extensively in the literature. For diffeomorphisms, a necessary condition is the presence of a \textit{dominated splitting} \cite{DPU1999, BDP2003}. More surprisingly, the diffeomorphism also has to be volume hyperbolic \cite{BDP2003}, which implies positive topological entropy \cite{Catsigeras-Tian}. Thus, since partially hyperbolic diffeomorphisms represent an important class of diffeomorphisms with dominated splitting, it is natural to explore conditions for robust transitivity for partially hyperbolic maps \cite{BV2000, PS2006, LP2013, CO2021, RHUY2022, PIN2025}.

We apply our proof of robust transitivity to a specific parametrized family of partially hyperbolic maps described below. More specifically, we prove that $f$ is partially hyperbolic and has a transitive hyperbolic subset, and that this set is associated with a blender. We prove that this blender is related in a particular way to a pair of fixed points $p_0$ and $q_0$ with unstable indices $1$ and $2$, and that the two-dimensional stable and unstable manifolds of $p_0$ and $q_0$, respectively, are dense. All these conditions imply that $f$ is $C^1$ robustly transitive, see \Cref{prop:rt-bdv}.

In a recent paper \cite{CKOZ-2025}, the authors developed a computer-assisted proof of the presence of blenders for diffeomorphisms in $\bbR^n$, and applied it to a three-dimensional H\'enon family of maps. Their characterization of the existence of blenders, given in \cite[Theorem 2]{CKOZ-2025}, is very general and does not use any geometrical information of the system \textit{a priori} for the theoretical result. On the other hand, when it comes to implementation, it is necessary to define explicitly a finite family of boxes that satisfy its hypotheses.

In comparison, our implementation of the existence of a blender (\Cref{sec:part2}) is somewhat tailored for the case where the unstable direction does not change much globally, since it uses curves aligned with the unstable cones and written as graphs on the $x$-coordinate. However, our result is closer to the definition, using open sets of curves. This makes the proof more straightforward. For the implementation, we instead have these open sets of curves defined automatically for a certain region, making it easier to change the parameters within a family without changing the code. A more refined version of this strategy appears in \cite{HMT2026}. We remark that the previous results in \cite{CKOZ-2025, HMT2026} are for systems in $\bbR^n$, while our results are stated for general manifolds, and implemented for the closed manifold $\bbT^3$.

The method to find a transitive hyperbolic subset associated with a blender is of particular interest. For this, we use symbolic dynamics and graph algorithms (see \Cref{sec:part1}).

Partial hyperbolicity is a statement about every point in the manifold, and its proof has the longest computation time of our method, as it checks cone inequalities at the tangent level for a cover of the manifold using small boxes (see \Cref{sec:part3}). Our strategy to prove the denseness of the two-dimensional stable/unstable manifolds for the fixed points also reflects the power of computer-assisted methods to prove properties for every point in a manifold, as we prove that every local unstable or stable leaf intersects a given two-dimensional invariant set (see \Cref{sec:part4}).

\subsection{Setting and statements}

To demonstrate our techniques for establishing robust transitivity in a setting that is partially hyperbolic but not Anosov, we look at a specific family of diffeomorphisms on the 3-torus.

Consider the family of matrices of the form
$$A_k = \left(\begin{matrix} 
    k & -1 & -1\\
    1 & 1 & 0\\
    1 & 0 & 0
\end{matrix}\right).$$

This family of matrices was studied by \cite{PT2014}, in which they proved that, for $k \geq 4$, $A_k$ defines an Anosov diffeomorphism with splitting $E^s \oplus E^u \oplus E^{uu}$, i.e., the unstable direction has dimension two and admits a further splitting into two one-dimensional directions, the first one being less expanding under $A_k$ than the second.

Ponce and Tahzibi applied to the action of $A_k^{-1}$ an abstract local perturbation (a Baraviera--Bonatti \cite{BB2003} perturbation), in order to create a partially hyperbolic map homotopic to $A_k^{-1}$, but with a different sign for the center Lyapunov exponent. Our goal is also to create partially hyperbolic maps homotopic to $A_k$ (or with inverse homotopic to $A_k^{-1}$), but using a global and explicit modification. With that, we explore the existence of blenders and robust transitivity for the new map with the validation of computer-assisted proofs.

The non-linear partially hyperbolic system is given by the composition of one map in the following family.

$$\psi_b(x,y,z) = (x, y - b \sin(2 \pi z), z).$$

We define
$$f_{k,b}(x,y,z) = \psi_b \circ A_k(x, y, z) = (kx -y - z, x + y - b \sin (2 \pi x), x).$$

We show in \Cref{sec:implementation} that the diffeomorphism $f_{k,b}$ is not Anosov so long as $b > \dfrac{1}{2\pi}$. For our analysis, we choose $b$ in a small interval centered at $1$, and then choose an integer $k$ large enough so that the system is partially hyperbolic. 

\begin{maintheorem}
\label{teoA}
For $k = 16$ and $b \in [0.9995, 1.0005]$, $f_{k,b}$ is partially hyperbolic and robustly transitive. Additionally, it has minimal strong foliations and heterodimensional cycles robustly.
\end{maintheorem}

In the above theorem, the strong foliations are the stable and unstable one-dimensional foliations of the partially hyperbolic splitting $E^s \oplus E^c \oplus E^u$. It is likely that the system is robustly transitive for many other choices of parameters, but as our main goal is to demonstrate the algorithm for establishing this property, we have chosen for simplicity and clarity to restrict our choice of parameters to those above.

The proof of Theorem A follows from Proposition \ref{prop:rt-bdv} --- which gives conditions for a partially hyperbolic diffeomorphism with a blender to be robustly transitive --- and \Cref{cor:consequences}. The main contributions of this work are: to state and prove results about pure mathematical objects that imply the hypotheses of \ref{prop:rt-bdv}; and to implement computational objects and algorithms that imply the desired pure mathematical objects and their properties. For the general statements, see \Cref{sec:main-statements}.

Partial hyperbolicity is verified numerically as part of the overall
algorithm.
For this system and choice of parameters,
it is possible to write a pen-and-paper proof that $f_{k,b}$
is partially hyperbolic.
However, the partially hyperbolic splitting must be computed and stored
by the computer in order to verify that the foliations are
minimal and that the overall system is robustly transitive.
Therefore, we only give a computer-aided proof of partial hyperbolicity.

We chose to use this specific family of systems on the 3-torus
as it is a tractable example of how the techniques and algorithms may be
applied. The diffeomorphism is defined as a composition of a linear map and a
shear, both of which are volume preserving.
There is a large existing body of pure mathematical results about partially
hyperbolic systems in dimension three, and even more results in the special cases where
the system is volume preserving and/or defined on the 3-torus.
In our algorithm, we do not make use of the fact that $\dim(M) = 3$ or that
the system is volume preserving, and the theory and computational techniques introduced
here are applicable to partially hyperbolic systems in any dimension.

A non-rigorous numerical study of the density of invariant manifolds for partially hyperbolic diffeomorphisms on $\bbT^3$ is given in \cite{GMK2019}.
That said, we now briefly discuss results related to Theorem A in the pure
mathematics literature. In \cite[Corollary D]{CP2026}, the authors prove that the minimality of the strong foliations is $C^1$-open and dense in the space of derived-from-Anosov diffeomorphisms in $\bbT^3$. Even with this genericity result, for the family of diffeomorphisms used in our implementation, we do not have the minimality of strong foliations \textit{a priori} (this is an announced result in \cite{ACEPWZ}: if $f: \bbT^3 \to \bbT^3$ is a volume-preserving $C^{1+\alpha}$ derived-from-Anosov diffeomorphism, then both foliations $W^u$ and $W^s$ are minimal), but the minimality of the strong foliations is a consequence of our results, which do not use volume preservation.

Since we define this family with a global deformation, we cannot use arguments similar to those originally appearing in \cite{mane1978}. And since we have mixed behavior in the center direction, we cannot use the Pujals--Sambarino \cite{PS2006} SH (\textit{some hyperbolicity}) property \textit{a priori}. The SH property is also a consequence: in \Cref{cor:consequences} we prove that every unstable segment has an iterate entering the robust covering collection (see \Cref{def:rcc}) defining the blender. This implies the SH property as defined in \cite[\S 5.1]{CP2026}.

In \cite[Theorem F]{yang21}, robust transitivity is guaranteed for partially hyperbolic diffeomorphisms that are volume preserving and whose center direction is one-dimensional and with non-zero Lyapunov exponents. We do not verify for which parameters in the family the system would have non-zero center Lyapunov exponents, since this is not necessary for our method.

Thus, even with the main strength of \Cref{teoA} being the method employed and not the conclusion itself, we remark that its conclusion is not a direct consequence of previous results.

\subsection{Structure of the paper}

\Cref{sec:preliminary} is dedicated to introducing notions in the theory of partially hyperbolic systems, and to stating \Cref{prop:rt-bdv} and its consequences. In \Cref{sec:main-statements} we define particular charts, cones, and submanifolds and certain properties related to how they behave under iteration of the dynamics. These objects are designed to bridge the gap between \Cref{prop:rt-bdv} and the computer-assisted strategy, using dynamical systems language.

In \Cref{sec:hset} we state, cite, or prove results used to guarantee that the algorithm we implement implies the hypotheses of the results in \Cref{sec:main-statements}. The implementation to verify the results from \Cref{sec:main-statements} for a particular system is given in \Cref{sec:implementation} in four parts: proving, for $f_{k,b}$, that a subset of $\bbT^3$ has a non-empty hyperbolic set $\Lambda$ with $\dim E^u = 2$ and that $\Lambda$ is transitive (\Cref{sec:part1}); proving that there is a family of curves $\mathcal{F}$ such that $(\Lambda, \mathcal{F})$ is a blender (see \Cref{def:blender}) and proving that $W^u(p_0)$ \textit{activates} the blender (\Cref{sec:part2}); proving partial hyperbolicity (\Cref{sec:part3}); proving that there is $r$ such that all strong unstable manifolds of points in $\bbT^3$ with length $r$ intersect $W^s(p_0)$ and all strong stable manifolds of points in $\bbT^3$ with length $r$ intersect $W^u(q_0)$ (\Cref{sec:part4}).

\section{Preliminary concepts}
\label{sec:preliminary}

We introduce in this section some notions and results regarding robust transitivity. We start with the concept of partial hyperbolicity, which is close to a more general necessary condition for robust transitivity for diffeomorphisms (the presence of a \textit{dominated splitting}). We then introduce the notion of a blender, which, as mentioned in the introduction, is one of the mechanisms to prove robust transitivity. To finish, we mention a criterion for robust transitivity in the partial hyperbolic context using blenders.

Throughout this work, $M$ is a closed (compact, connected, finite dimensional and without boundary) Riemannian manifold and $f: M \to M$ is a $C^1$ diffeomorphism.

\subsection{Partial hyperbolicity}
\label{sec:ph}

We say that a diffeomorphism $f: M \to M$ is \textit{partially hyperbolic} if there is a $Df$-invariant splitting of the tangent bundle $TM = E^s \oplus E^c \oplus E^u$ such that
$$\Vert Df_pv^s \Vert < \Vert Df_pv^c \Vert < \Vert Df_pv^u \Vert \quad \text{and} \quad\Vert Df_pv^s \Vert < 1 < \Vert Df_pv^u \Vert$$
for all $p \in M$ and all unit vectors $v^{\sigma} \in E^{\sigma}_p$.

A diffeomorphism is \textit{Anosov} or \textit{uniformly hyperbolic} if $E^c$ is trivial, meaning that we have an invariant splitting of the tangent space into a direction of uniform expansion and a direction of uniform contraction. Given an open set $U \subseteq M$, we say that a compact $f$-invariant $\Lambda \subseteq U$ is a \textit{uniformly hyperbolic set} if we have such splitting for points in $\Lambda$.

Remember that a cone $\Cone$ in $\bbR^n$ is a set of vectors satisfying $v \in \Cone$ implies $a v \in \Cone$ for all $a \in \bbR$. In this work we consider, in particular, two kinds of cones, see \Cref{sec:cones} for their definitions.

We consider a \textit{cone family} in $M$ as $\Cone = \{\Cone(p)\}_{p \in M} \subseteq TM$ with each $\Cone(p) \subseteq T_pM$ being a cone. We say that $\Cone$ is \textit{$Df$-invariant} if $Df(p)(\Cone(p)) \subseteq \textrm{Int } \Cone(f(p))\cup \{0\}$. An equivalent definition of partial hyperbolicity is given by invariant contracting and expanding cones as follows.

\begin{definition}
    \label{def:ph}
    We say that a diffeomorphism $f: M \to M$ is \textit{partially hyperbolic} if there are two cone families $\Cone^u$ and $\Cone^s$ in $M$ satisfying:
    \begin{enumerate}
        \item $\Cone^u$ is $Df$-invariant;
        \item $\Cone^{s}$ is $Df^{-1}$-invariant;
        \item $\Vert Df_pv^u \Vert > 1 \mbox{ for $p \in M$ and all unit vectors } v^u \in\Cone^u(p);$
        \item $\Vert Df^{-1}_pv^s \Vert > 1 \mbox{ for $p \in M$ and all unit vectors } v^s \in\Cone^s(p).$
    \end{enumerate}
\end{definition}

If (1) and (3) hold, then $f$ has the splitting $E \oplus E^{u}$, with $E^u$ being uniformly expanded. If (2) and (4) hold, then $f$ has the splitting $E^s \oplus F$, with $E^s$ being uniformly contracted. Thus, if all conditions are satisfied, $f$ is partially hyperbolic.

We can express this definition of partial hyperbolicity using cones in local coordinates, creating the following criterion.

\begin{proposition}
    \label{prop:1ph-proof}
    Consider an atlas $\{ (B_i, \; \varphi_i: V_i \subseteq \bbR^n \to B_i) \}_{i \in \{0, \; \cdots, \;m-1\}}$ for the n-dimensional manifold $M$ and a diffeomorphism $f: M \to M$.
    Consider, for each $i$, cone fields $\Cone^u_i$ and $\Cone^s_i$ over $\varphi_i^{-1}(B_i)$. Suppose that, for every $i, j \in \{0, \; \cdots, \;m-1\}$ such that $f(B_i) \cap B_j \neq \varnothing$, the following conditions hold:
    \begin{enumerate}
        \item \emph{u-invariance}: $Df_{ij}(p) (\Cone^u_i(p)) \subseteq \textrm{Int } \Cone^u_j(f(p))\cup \{0\}$;
        \item \emph{s-invariance}: $Df_{ij}^{-1}(p) (\Cone^s_j(p)) \subseteq \textrm{Int } \Cone^s_i(f^{-1}(p)) \cup \{0\}$;
        \item \emph{u-expansion}: $\Vert Df_{ij}(p) \cdot v \Vert > \Vert v \Vert$ for all $v \in \Cone^u_i(p)$;
        \item \emph{s-contraction}: $\Vert Df_{ij}^{-1}(p) \cdot v \Vert > \Vert v \Vert$ for all $v \in \Cone^s_j(p)$,
    \end{enumerate}
    where $p \in \varphi_i^{-1}(B_i \cap f^{-1}(B_j))$ and $f_{ij} = \varphi_j^{-1} \circ f \circ \varphi_i$ is $f$ in local coordinates. Then $f$ is partially hyperbolic.
\end{proposition}

\begin{proof}
    Define $\Cone^u(q) := \bigcup_{\{i \; : \; q = \varphi_i(q_i)\}} D\varphi_i(q_i) (\Cone^u_i(q_i))$ as an unstable cone for $q \in M$, and define analogously the stable cones. Then these cones satisfy the conditions for partial hyperbolicity.
\end{proof}

For partially hyperbolic maps, using graph transform methods as in \cite{HPS}, we can prove that there are continuous families of foliations $W^s$ and $W^u$ tangent respectively to the bundles $E^s$ and $E^u$.

We say that a map $f: \bbT^n \to \bbT^n$ is \textit{derived-from-Anosov} if it is homotopic to an Anosov map. The Anosov map can be given by a linear Anosov map, induced by $A$ a $n \times n$ matrix with integer entries, $\vert \!\det A \vert = 1$ and with eigenvalues outside the complex unit circle. The linear toral map $A$ induces the same action as $f$ on $\bbZ^n = \pi^1(\bbT^n)$, and it is called the \textit{linearization} of $f$.

Historically, these were one of the first few cases of partially hyperbolic maps to be studied, and their relationship with a linear Anosov map allowed, for instance, the proof of robust transitivity for some examples inside this class \cite{mane1978,shub71}.

The family of examples that we apply our results to belongs to this class, but this is not a requirement for our methods.

\subsection{Blender}
\label{sec:blender}

A \textit{k-dimensional plaque}, in the context of this work, means a $C^1$ embedded image in $M$ of the unit disk in $\bbR^k$. We introduce the notion of a \textit{robust covering collection} for an n-dimensional manifold $M$ as a family of k-dimensional plaques that present a form of robust forward-invariance. More precisely, we define it as follows.

\begin{definition}
    \label{def:rcc}
    We say that a set $\Fcal$ of $C^1$ k-dimensional plaques in $M$, $1 \leq k < n$, is a \emph{robust covering collection} for $f$ if there is a $C^1$-neighborhood $\mathcal{U}$ of $f$ satisfying: for every $g \in \mathcal{U}$ and every $\al \in \Fcal,$ there exists a subset $\bt \subof \al$ such that $g(\bt) \in \Fcal.$ We say that $\dim(\Fcal) = k$.
\end{definition}

Note that the presence of a robust covering collection for the diffeomorphism $f$ is a robust property. Additionally, $\Fcal$ can be considered to be $C^1$-open. Indeed, if $\mathcal{U}' \subseteq \mathcal{U}$ is $C^1$-open and $\mathcal{H}$ is a small enough neighborhood of $\operatorname{Id}$ such that $h \inv \circ g \circ h \in \mathcal{U}$ for all $h \in \mathcal{H}$ and $g \in \mathcal{U}'$, then the set
$$\Fcal' = \{h \circ \al \, : \, h \in \mathcal{H} \quad \text{and} \quad \al \in \Fcal \},$$
is $C^1$-open and satisfies the definition of a robust covering collection.

There are several related notions of blender in the literature. In general, the blender is considered together with a hyperbolic fixed point, or a uniformly hyperbolic set that may also be required to be transitive. This additional structure implies properties that are commonly associated with a blender.

\begin{definition}
    \label{def:blender}
    We say that the pair $(\Lambda, \mathcal{F})$ is a \textit{blender} (or \textit{cu-blender}) if $\Lambda$ is a transitive invariant set that is uniformly hyperbolic with $\dim \restr{E^u}{\Lambda} > k$, and there are a $C^1$-neighborhood $\mathcal{U}$ of $f$ and a set $\mathcal{F}$ of k-dimensional plaques such that
    $$\gamma \cap W^s(\Lambda_g) \neq \varnothing,$$
    for every $\gamma \in \mathcal{F}$, and every $g \in \mathcal{U}$, where $\Lambda_g$ is the hyperbolic set of $g$ that is a continuation of $\Lambda$.
    
    An index-k hyperbolic periodic point for $f$ is said to \textit{activate} the blender if its unstable manifold contains a sub-plaque belonging to $\mathcal{F}$.
\end{definition}

We recall the concept of uniform hyperbolicity in \Cref{sec:ph}. We say that $f: M \to M$ is \emph{topologically transitive} (or simply \textit{transitive}) if, for all $U, V \subseteq M$ non-empty open sets, there is $k \in \bbN$ such that $f^k(U) \cap V \neq \varnothing$. If $M$ is a closed manifold, transitivity is equivalent to the existence of a dense orbit. We call an invariant subset $\Lam = f(\Lam) \subseteq M$ transitive if the restriction $\restr{f}{\Lam}$ is transitive.

As with \Cref{def:rcc}, the presence of a blender is a robust property in the $C^1$ topology.

In particular, considering the set $V = \bigcup_{\al \in \Fcal} \al$, if $\Lam = \bigcap_{n \in \bbZ} f^n(V)$ is a uniformly hyperbolic set for $f$ with $\dim \restr{E^u}{\Lambda} > k$ and $\Fcal$ is a k-dimensional robust covering collection for $f$, then $(\Lambda, \mathcal{F})$ is a blender. The proof of this fact is in \cite[Proposition 2.5]{HMT2026} and it is also contained within the proof of \Cref{prop:blender}.

With \Cref{def:blender} we can recover the general meaning of blender as

\begin{quote}
    \textit{``(...)a blender is a hyperbolic basic set such that convenient projections of its stable, or unstable, set have larger topological dimension than the stable, or unstable, set itself.''} \cite{BDV2005}
\end{quote}

Consider, for simplicity, the case that $M$ is a 3-manifold, $k = 1$ and $\dim \restr{E^u}{\Lambda} = 2$. Intuitively,  since $\mathcal{F}$ is $C^1$-open, then its intersection with a transversal plane is an open subset of this plane. This means that the projection of $W^s(\Lambda)$ along the curves in $\mathcal{F}$ ``fills'' an open set of this plane, which resembles the defining property of a blender introduced in \cite{HKOS2018} as the \textit{carpet property}. 

\subsection{Robust transitivity}
\label{sec:rt}

We are interested in the persistence of transitivity under perturbation.

\begin{definition}
    \label{def:rt}
    Given a Riemannian manifold $M$, we say that a diffeomorphism $f: M \to M$ is ($C^1$-)\textit{robustly transitive} if there is a neighborhood $\mathcal{U} \supset f$ in $\operatorname{Diff^1(M)}$ with the $C^1$ topology such that, for all $g \in \mathcal{U}$, $g$ is transitive.
\end{definition}

For uniformly hyperbolic diffeomorphisms, robust transitivity is a consequence of their structural stability, meaning that all topological properties are persistent on a neighborhood of $f$. Other mechanisms used to prove robust transitivity frequently borrow topological properties from Anosov maps. For instance in \cite{mane1978}, the author consider derived-from-Anosov maps (see \Cref{sec:ph}) coming from local perturbations of Anosov maps, and in \cite{PS2006} the authors prove the stronger property of robust minimality of stable/unstable foliations (meaning that all leaves are dense and this occurs robustly), something that is also inspired by the topological behavior of the Anosov case.

With the introduction of blenders (see \Cref{sec:blender}), more general examples of robustly transitive systems were constructed. The strategy in \cite{bonatti-diaz} to prove robust transitivity using blenders is to ``connect'' the blender with the invariant manifolds of two different hyperbolic fixed/periodic points. Using this ``connection'', we can show that one of these hyperbolic periodic points have both its stable and unstable manifold being dense, which implies transitivity, and this property is robust. This result is \cite[Proposition 7.4]{BDV2005}, and we present it here as \Cref{prop:rt-bdv}.

In the following, we use $W^{s/c/u}$ to refer to the invariant foliations given by the partially hyperbolic splitting. An \textit{unstable disk} is an embedded disk that lies inside a leaf of the partially hyperbolic unstable foliation. An \textit{unstable disk of radius $r > 0$} consists of all points $x$ with $d_u(x_0,x) \leq r$, where $x_0$ is the center of the disk and $d_u$ is distance measured along $W^u(x_0)$. A \textit{stable disk} (of radius $r$) is defined analogously. We use $\hypW^{s/u}(p)$ to refer to the stable/unstable set of a hyperbolic periodic point. If $Df(p)$ restricted to $E^c(p)$ is not contracting, then $\hypW^{s}(p) = W^{s}(p)$. In general, it only holds that $\hypW^{s}(p) \subsetneq W^{cs}(p)$. The same holds for unstable leaves.

\begin{proposition}
    \label{prop:rt-bdv}
    Consider $M$ an n-dimensional closed manifold and $f: M \to M$ a diffeomorphism. If
    \begin{enumerate}
    \item $f$ is partially hyperbolic with splitting $TM = E^s \oplus E^c \oplus E^u$, where the subbundles have dimensions $s, c, u \geq 1$, respectively;
    \item there are $p_0, q_0 \in M$ hyperbolic periodic points with $\dim \hypW^u(p_0) = u + c_1$ and $\dim \hypW^u(q_0) = u+c_2$, where $0 \leq c_1 < c_2 \leq c$;
    \item $q_0$ is homoclinically related to a blender $(\Lambda, \mathcal{F})$, where $\Lambda$ has unstable index $u+c_2$ and the plaques in $\mathcal{F}$ have dimension $u+c_1$;
    \item $\hypW^u(p_0) = W^u(p_0)$ activates the blender $\mathcal{F}$, meaning that $W^u(p_0)$ has a subset belonging to $\mathcal{F}$;
    \item there is a constant $r > 0$ such that any unstable disk of radius larger than $r$ intersects $\hypW^s(p_0)$ and any stable disk of radius larger than $r$ intersects $\hypW^u(q_0)$.
    \end{enumerate}
    Then $f$ is robustly transitive.
\end{proposition}

We recall that, in the item (3) above, $q_0$ being \textit{homoclinically related} to the hyperbolic set $\Lambda$ means that they have the same unstable index and there is $p \in \Lambda$ a periodic point such that
\begin{align*}
    \hypW^{s}(q_0) &\pitchfork \hypW^{u}(p) \neq \varnothing \\
    \hypW^{s}(p) &\pitchfork \hypW^{u}(q_0) \neq \varnothing.
\end{align*}

In the setting of our implementation, we have $M = \bbT^3$, $c_1 = 0$ and $c_2 = u = s = 1$.

An immediate consequence of \Cref{prop:rt-bdv} is the following result.

\begin{corollary}
    \label{cor:consequences}
    A diffeomorphism $f: M \to M$ satisfying the five conditions from \Cref{prop:rt-bdv} also satisfies:
    \begin{itemize}
        \item $f$ has a \textit{heterodimensional cycle};
        \item the strong foliations of $f$ are minimal;
        \item every unstable disk has an iterate which contains a plaque in $\Fcal.$
    \end{itemize}
    Additionally, these properties are robust.
\end{corollary}

\begin{proof}  
    The hyperbolic periodic point $p_0$ and the transitive hyperbolic set $\Lambda$ have different unstable indexes, and Conditions (4) and (5) imply, respectively, that 
    \begin{align*}
        \hypW^{u}(p_0) &\cap \hypW^{s}(\Lambda) \neq \varnothing\\
        \hypW^{u}(\Lambda) &\cap \hypW^{s}(p_0) \neq \varnothing,
    \end{align*}
    thus, $f$ has a heterodimensional cycle.

    Condition (5) of \Cref{prop:rt-bdv} implies that $\hypW^u(q_0)$ is dense. Applying the lambda lemma, we have that any open set $U$ contains a disk $D$ whose image $f^n(D)$ approximates
    the local unstable manifold of $q_0$, since it is a fixed point.
    Any stable disk $\gamma$ of the partially hyperbolic splitting with radius bigger than $r$ also
    intersects the local unstable manifold of $q_0$, and hence $\gamma$ intersects $f^n(D)$ for all $n$ sufficiently large.

    Finally, to verify that any unstable disk $J$
    has an iterate which contains a curve in $\Fcal,$ note that any open set $U$ has a disk $D$ whose preimage $f \invn(D)$ contains
    close approximations of every local unstable manifold in $\Lam.$
    This implies the forward orbit of $J$ eventually hits $U.$
\end{proof}

By checking the hypotheses of \Cref{prop:rt-bdv}, we obtain the conclusion of \Cref{teoA}. The consequences stated in \Cref{cor:consequences} are not specific to the three-dimensional case. 

\section{Criterion for robust transitivity}
\label{sec:main-statements}

In this section we describe conditions that imply the hypotheses of \Cref{prop:rt-bdv}. These conditions involve covering the manifold with hypercubes equipped with cones and checking properties of their intersections under iteration. In particular:
\begin{enumerate}
    \item \Cref{prop:weakph} gives us partial hyperbolicity;
    \item \Cref{cor:fixedpoint} allows us to check for fixed points $p_0$ and $q_0$ with indices $u + c_1$ and $u + c_2$ respectively;
    \item we use \Cref{prop:lam} to find a transitive and hyperbolic set $\Lam$ containing $q_0$, and \Cref{prop:blender} gives us conditions for a family of plaques, together with $\Lam$, to be a blender;
    \item \Cref{lemma:activate} allows us to prove that $\hypW^u(p_0)$ activates the blender;
    \item \Cref{prop:densestable} is used to prove the robust denseness of $\hypW^s(p)$ and $\hypW^u(q)$.    
\end{enumerate}

These results mentioned above are stated and proved in this section analytically using usual partial hyperbolicity notation and arguments, but they can be checked computationally in a rigorous way.

\subsection{Dynamical boxes, u-plaques and covering relation}
\label{sec:dynBox}

We define strong dynamical boxes and strong covering relations in an abstract setting. For the specific definition used for implementation purposes, see \Cref{sec:coneBox}.

\begin{definition}
    \label{def:strongDynBox}
    A \emph{strong dynamical box} $B$ is the image of
    a $C^1$ embedding
    \[
        \varphi : [-1,1]^u \times [-1,1]^s \to M
    \]
    with $u + s = \dim(M)$ and
    equipped with cone families $\Cone^u$ and $\Cone^s$ defined on the cube $\cube = [-1,1]^u \times [-1,1]^s$
    (which define cones in $TM$ by composition with $D\varphi$) satisfying:
    \begin{itemize}
    \item the unstable cone family $\Cone^u$ is of dimension $u$,
    transverse to $\{x\} \times [-1,1]^s$ for all $x$, and contains $\bbR^u \times \{0\}$;
    \item the stable cone family $\Cone^s$ is of dimension $s$,
    transverse to $[-1,1]^u \times \{y\}$ for all $y$, and contains $\{0\} \times \bbR^s$;
    \item vectors in $D\varphi \, \Cone^u$ are
    uniformly expanded by $Df$ and vectors in $D\varphi \, \Cone^s$ are uniformly contracted by $Df.$
    \end{itemize}
\end{definition}

A strong dynamical box expresses in local coordinates information that can be used
to check contraction and expansion, both in the manifold level (in local coordinates, with $[-1,1]^u \times [-1,1]^s$), and in the tangent space level, with the cones. To make this precise, we first define u and s-plaques as follows.

\begin{definition}
    \label{def:u-plaque}
    A $C^1$ embedded disk $\rho \subof B$ is a \emph{u-plaque} if it is
    the image under $\varphi$ of a graph of a function $\rho_{\loc}: [-1,1]^u \to [-1,1]^s$
    that is tangent to $\Cone^u$.
    
    Similarly, a $C^1$ embedded disk $\rho \subof B$ is an \emph{s-plaque} if it is
    the image under $\varphi$ of a graph of a function $\rho_{\loc}: [-1,1]^s \to [-1,1]^u$
    that is tangent to $\Cone^s$.
    
    If $u$ or $s = 1,$ then we may refer to u/s-plaques as \emph{u/s-curves}.
\end{definition}

Trivially, $\varphi([-1,1]^u \times \{y\})$ is a u-plaque for all $y$, and $\varphi(\{x\} \times [-1,1]^s)$ is an s-plaque for all $x$.

For two cones $\Bone, \Cone \subset V$ in a vector space $V,$
we say that $\Bone$ lies in the interior of $\Cone$ if every
non-zero vector in $\Bone$ lies in the interior of $\Cone$.

\begin{definition}
    \label{def:strongCov}
    Given $B_i$ and $B_j$ strong dynamical boxes, and consider $\floc = \varphi_j^{-1} \circ f \circ \varphi_i$.
    We say that $B_i$ and $B_j$ satisfy
    a \emph{strong covering relation}, denoted $B_i \fcovers B_j$, if:
    \begin{enumerate}
        \item For every point $p \in \cube$ with image $\floc(x) \in \cube$,
        the cone $D\floc(\Cone^u_i(p))$ is contained in the interior
        of $\Cone^u_j(\floc(p))$ and the cone $D\floc \inv(\Cone^s_j(\floc(p)))$ is contained in the interior
        of $\Cone^s_i(\varphi_i \inv(p)).$
        \item
        For every u-plaque $\al$ of $B_i,$ there is a unique subplaque
        $\bt$ in the interior of $\al$ such that $f(\bt)$ is a u-plaque of $B_j.$
        \item
        For every s-plaque $\al$ of $B_j,$ there is a unique subplaque
        $\bt$ in the interior of $\al$ such that $f\inv(\bt)$ is an s-plaque of $B_i.$
    \end{enumerate}
\end{definition}
This condition can be verified by defining the cone families
using quadratic forms and rigorous numerics.
More details are given in \Cref{sec:hset}, including how the above definition relates to similar concepts in the literature.

We now define a weaker covering relation $B_i \fweakcovers B_j$ which
is easier to verify computationally than the strong covering relation
$B_i \fcovers B_j$ defined above and which suffices for many parts
of the overall proof.
\begin{definition}
    \label{def:weakDynBox}
    A \emph{weak dynamical box} $B$ is the image of
    a $C^1$ embedding
    \[
        \varphi : [-1,1]^u \ti [-1,1]^s \to M
    \]
    with $u + s = \dim(M)$ and
    equipped with a cone family $\Cone^u.$
    The unstable cone family $\Cone^u$ is of dimension $u$
    and transverse to $\{x\} \ti [-1,1]^s$ for all $x.$
\end{definition}

As in the case of strong dynamical boxes, a $C^1$ embedded disk $\rho \subof B$ is a \emph{u-plaque} if it is tangent to $\Cone^u$
and the image under $\varphi$ of a graph of a function $[-1,1]^u \to [-1,1]^s.$
We remark that this definition has no stable cone family
or s-plaques and no requirement of uniform expansion for the unstable cones.
Note also that any strong dynamical box can be regarded as a weak dynamical
box.

\begin{definition}
    \label{def:weakCov}
    Two weak dynamical boxes satisfy
    a \emph{weak covering relation} $B_i \fweakcovers B_j$ if:
    \begin{enumerate}
        \item For every point $x \in B_i$ with image $f(x) \in B_j,$
        the cone $D\floc(\Cone^u_i(\varphi_i \inv(x)))$ is contained in the interior
        of $\Cone^u_j(\floc(\varphi_j \inv(x)))$.
        \item
        For every u-plaque $\al$ of $B_i,$ there is a subplaque
        $\bt$ in the interior of $\al$ such that $f(\bt)$ is a u-plaque of $B_j.$
    \end{enumerate}
\end{definition}

Observe that, unlike the definition of strong covering relation (\Cref{def:strongCov}), we do not require the uniqueness of the subplaque $\bt$ in the above definition.

For a collection of weak dynamical boxes $\{ B_i \},$
we say that the boxes have \emph{compatible unstable cones}
if for any $x \in B_i$ with image $f(x) \in B_j,$
the derivative $D\floc$ maps $\Cone^u_i(x)$ into the interior
of $\Cone^u_j(\floc(x)).$
Note that this is even weaker than the weak covering relation
as there are no requirements on the u-plaques.

\begin{proposition} \label{prop:weakph}
    Let $\{B_i\}$ be a covering of $M$ by weak dynamical boxes
    with compatible unstable cones
    and assume that the unstable cones are uniformly
    expanded by $Df.$
    Then, $f$ has a weak partially hyperbolic splitting
    $\Ecs \oplus \Eu.$
\end{proposition}

The above is a direct consequence of \Cref{prop:1ph-proof} applied only for $\Cone^u$.

\subsection{Hyperbolicity and transitivity}
\label{sec:hyp-trans}

To prove transitivity, we use the concept of strong covering defined in \Cref{sec:dynBox}.

An \emph{adjacency matrix} is a square matrix $T$ with each entry
either zero or one. We allow both zeros and ones on the diagonal.
Such a matrix $T$ defines a subshift of finite type
$\sigma : \Sig_T \to \Sig_T$
where the elements of $\Sig_T$ are sequences of the form
$\{a_n\}_{n \in \bbZ}$ where $T_{a_n, a_{n+1}} = 1$ for all $n \in \bbZ.$
There are well-established algorithms which, given a matrix $T,$
determine whether or not the resulting subshift is transitive.

\begin{theorem} \label{thm:coding}
    Suppose $\{B_1, \ldots, B_m\}$ is a finite collection of
    strong dynamical boxes and $T$ is an $m \times m$ incidence matrix
    such that if $T_{ij} = 1$ then $B_i \fcovers B_j.$

    Then, for any sequence $\{a_n\}$ in $\Sig_T,$
    there is a unique point $x \in M$ such that
    $f^n(x) \in B_{a_n}$ for all $n \in \bbZ.$
    This defines a map $h : \Sig_T \to M$ such that
    $h \circ \sig = f \circ h$ and
    the image $h(\Sig_T)$ is uniformly hyperbolic.
    For a point $x = h(\{a_n\})$ in the image:
    \begin{itemize}
        \item the intersection
        $\bigcap_{n \ge 0} f \invn(B_{a_n})$ is equal to the connected component
        of $\hypW^s(x) \cap B_{a_0}$ through $x$ and
        \item
        the intersection
        $\bigcap_{n \le 0} f \invn(B_{a_n})$ is equal to the connected component
        of $\hypW^u(x) \cap B_{a_0}$ through $x.$
    \end{itemize}
\end{theorem}

The above result is similar, but not identical, to existing results in symbolic dynamics and rigorous numerics \cite{Zglic2009,CKOZ-2025}, thus we give a proof in \Cref{appa}.

\begin{remark}
    \label{rem:zglic-stable-manifold}
    In \cite{Zglic2009}, the author proves the stable manifold theorem for fixed points using a covering relation (see \Cref{def:covering}) and a cone condition at the manifold level (see \Cref{eq:zglic-cone}), that imply our notion of strong covering relation, see \Cref{lem:covRel-quadCone-imply-strongCov}. Using ideas from that work, it is possible to give an alternative proof of the above result with stronger hypotheses, see \Cref{rem:coding-references}.
\end{remark}

Note that if the subshift $\sig : \Sig_T \to \Sig_T$ is transitive,
it implies that the image $h(\Sig_T)$ is a transitive hyperbolic set for $f.$

For our construction of a blender, we wish to have a 
transitive and uniformly hyperbolic set $\Lam$ and 
a set $V,$ with the property
that if $y \in \bigcap_{n \ge 0} f \invn (V)$
then $y \in \hypW^s(x)$ for some point $x \in \Lam.$
Then $V$ is a candidate region to construct the family $\mathcal{F}$
in the definition of the blender over the hyperbolic set $\Lam$.
In our case, we define the uniformly hyperbolic set
$h(\Sig_{T})$ containing $\bigcap_{n \in \bbZ} f^n(V)$.

Starting from a set $V$ subset of $M,$
we express $V$ as a finite union of sets $s_i$ of small diameter.
In the specific case of our system on the 3-torus,
the $s_i$ are small axis-aligned boxes.
We cover each $s_i$ by a much larger strong dynamical box $B_i,$
and we say that $s_i$ is a \textit{seed} for $B_i$.
We then use \Cref{thm:coding} to establish uniform hyperbolicity
and transitivity.

\begin{proposition} \label{prop:lam}
    Suppose $\{B_1, \ldots, B_m\}$ is a finite collection of strong
    dynamical boxes and $\{s_1, \ldots, s_m\}$ the corresponding
    collection of seeds.
    Suppose $T$ is an incidence matrix such that:
    \begin{enumerate}
        \item if $f(s_i)$ intersects $s_j,$ then $T_{ij} = 1,$
        \item
        if $T_{ij} = 1,$ then $B_i \fcovers B_j,$ and
        \item the matrix $T$ defines a strongly connected graph.
    \end{enumerate}
    Let $\sig : \Sig_{T} \to \Sig_{T}$ and $h : \Sig_{T} \to M$ be as in \Cref{thm:coding}.
    Then $h(\Sig_{T})$ is transitive and any point
    \[
        y \in \bigcap_{n \ge 0} f \invn(s_1 \cup \cdots \cup s_m) = \bigcap_{n \ge 0} f \invn(V)
    \]
    lies in the stable manifold of a point in $h(\Sig_{T}).$
\end{proposition}

\begin{proof}
    \Cref{thm:coding} applied to $T$ gives us the semiconjugacy $h$, and the strong connectivity assumption gives us that $\sigma: \Sig_{T} \to \Sig_{T}$ is transitive.

    For any $y \in \bigcap_{n \ge 0} f \invn(V)$, there is a sequence $\{a_n\}_{n \geq 0} $ such that $f^n(y) \in B_{a_n}$ for $n \geq 0$. By the strong connectivity assumption, we can extend it to a bi-infinite sequence $\{a_n\} \in \Sigma_T$. By setting $x = h(\{a_n\})$, we have that $y \in \hypW^s(x)$ by \Cref{thm:coding}.
\end{proof}

We apply \Cref{prop:lam} to our system on the 3-torus to
establish a uniformly hyperbolic set $\Lam = h(\Sig_{T})$
with two-dimensional unstable direction and a subset $V \subof \bbT^3$
such that any point $x \in \bigcap_{n \ge 0} f \invn(V)$
lies in the one-dimensional stable manifold of a point in $\Lam.$
The set $V$ is of the form $V_x \ti S^1 \ti V_z$, see \Cref{sec:implementation} for the details.

\subsection{Blender}
\label{sec:blender-crit}

With the transitive uniformly hyperbolic set $\Lam$ constructed,
the next step is to use it to find a blender.

\begin{proposition} \label{prop:blender}
Suppose $\Fcal$ is a robust covering collection,
$\Lam_f$ is a transitive uniformly hyperbolic set for $f$,
$\Ucal$ is a $C^1$-neighborhood of $f,$
and $K$ is a subset of $M$ with the following properties:
\begin{enumerate}
\item $\dim(\Fcal) + \dim W^s(\Lam_f) < \dim(M),$
\item
all elements of $\Fcal$ are contained in $K,$ and 
\item
if $g \in \Ucal,$ then $\bigcap_{n \in \bbZ} g^n(K)$ is contained in the
hyperbolic continuation $\Lam_g$ of $\Lam_f.$
\end{enumerate}
Then, $(\Lam_f, \Fcal)$ is a blender.
\end{proposition}
\begin{proof}
Without loss of generality, assume $\Ucal$ is also the neighbourhood of $f$
used in the definition of a robust covering collection.
Consider $g \in \Ucal$ and $\al_0 \in \Fcal.$
Then there is a sequence of plaques $\al_n \in \Fcal$ such that
$\al_{n+1} \subof g(\al_n)$ for all $n \ge 0.$
As the plaques are compact,
the nested intersection $\bigcap_{n \ge 0} g \invn(\al_n)$ is non-empty.
Let $p$ be a point in this intersection. Then $p \in \bigcap_{n \ge 0} g \invn(K).$

Define $K_n = \bigcap_{j \ge -n} g^{-j}(K)$ and note that
$g^n(p) \in K_n$ for all $n \ge 0.$
As $K$ is compact, $g^n(x)$ has a convergent subsequence $g^{n_k}(p)$
whose limit lies in $\bigcap_{n \ge 0} K_n \subof \Lam_g.$
In fact, any subsequence of $g^n(p)$ has a sub-subsequence
that converges to a point in $\Lam_g$ and so
$\lim_{n \to \infty} d(g^n(x), \Lam_g) = 0.$
This implies that $p \in W^s(\Lam_g).$
\end{proof}

As mentioned in the previous subsection,
for our example system $f = f_{k,b}$ on the 3-torus,
the transitive hyperbolic set $\Lam_f$ is defined on a superset of
$\bigcap_{n \in \bbZ} f^n(V)$ where $V = V_x \ti S^1 \ti V_z \subof \bbT^3.$
For this system,
the collection $\Fcal$ will consist of graphs of functions $V_x \to S^1 \ti S^1$
satisfying certain additional properties explained in \Cref{sec:part2}
which will be used numerically to verify that $\Fcal$ is a robust covering
collection.
These graphs are all subsets of $K = V_x \ti S^1 \ti S^1 \subof \bbT^3.$

In the formula for $f = f_{k,b},$
the $x$-coordinate in one iteration becomes the
$z$-coordinate at the next iteration.
Using this and taking the interval $V_z$ to be slightly larger than $V_x,$
if $g : \bbT^3 \to \bbT^3$ is $C^1$-close to $f,$
then $g(V_x \ti S^1 \ti S^1) \subof S^1 \ti S^1 \ti V_z$
and therefore $\bigcap_{n \in \bbZ} g^n(K) = \bigcap_{n \in \bbZ} g^n(V).$
The strong covering condition used to prove uniform hyperbolicity
is a $C^1$-open condition, and so $\bigcap_{n \in \bbZ} g^n(V)$ is contained in the
hyperbolic continuation $\Lam_g$ for all $g$ sufficiently close to $f.$
Proposition \ref{prop:blender} applies to show that
$(\Lam_f, \Fcal)$ is a blender.

We could have used $K$ directly to construct $\Lam_f,$
but this would have required that the computer program construct more
boxes than necessary, and so we construct the set using $V.$
Implementation details for the collection $\Fcal$ are given in \Cref{sec:part2}.

\subsection{Activating the blender}
\label{sec:activation-crit}

Next, we establish that there is a hyperbolic fixed point
$p_0$ with one-dimensional unstable manifold and which activates
the blender $(\Lam, \Fcal).$
Towards this aim, we use the following corollary of \Cref{thm:coding}.

\begin{corollary} \label{cor:fixedpoint}
    If $B$ is a strong dynamical box such that $B \fcovers B$
    then $\bigcap_{n \in \bbZ} f \invn(B)$ consists a single
    fixed point $p.$
    Moreover, $p$ is hyperbolic, $\hyp{B}^u(p)$ is a u-plaque
    and $\hyp{B}^s(p)$ is an s-plaque.
\end{corollary}

The notation $\hyp{B}^s(p)$ above is used for the connected component of $\hypW^s(p) \cap B$ through the fixed point $p$, and analogously for the unstable one.

A similar corollary can be stated which establishes the existence
of hyperbolic periodic orbits, but we do not need it for our specific
construction.

Another consequence of \Cref{thm:coding} is the following.

\begin{corollary}
    \label{cor:transv-inter}
    If $B$ is a strong dynamical box such that $B \fcovers B$, with $p$ its unique fixed point, then every u-plaque in $B$ transversely intersects $\hyp{B}^s(p)$. Analogously, every s-plaque in $B$ transversely intersects $\hyp{B}^u(p)$.
\end{corollary}

\begin{proof}
    If $\gamma_0$ is a u-plaque for $B$, then it has a subcurve $\gamma_1 \subseteq \gamma_0$ such that $f(\gamma_1)$ is a u-plaque for $B$. Inductively, we have that $\gamma_0$ contains a point $z$ whose orbit under $f$ is always inside $B$, that is, $f^{k}(z) \in B$ for all $k \geq 0$. By \Cref{thm:coding}, $\hypW^s(p) \cap B = \bigcap_{k = 0}^{\infty} f^{-k} (B)$, thus $z \in \hypW^{s}(p) \cap \gamma_0$. The hyperbolicity implies that this intersection is transversal.
\end{proof}

For a blender $(\Lam, \Fcal)$ and a weak dynamical box $B,$
we also introduce the notation $B \fweakcovers \Fcal$
to denote that every u-plaque $\al$ of $B$ contains a subplaque $\bt$
in its interior such that $f(\bt) \in \Fcal.$

\begin{lemma} \label{lemma:activate}
    If $B_0, B_1, \ldots, B_l$ are dynamical boxes and
    $\Fcal$ is a blender such that
    \[
        B_0 \fcovers B_0 \fweakcovers B_1
        \fweakcovers \cdots \fweakcovers B_l
        \fweakcovers \Fcal,
    \]
    then the unique hyperbolic fixed point in $B_0$ activates the blender.
\end{lemma}

\begin{proof}
    Since $B_0 \fcovers B_0$, \Cref{cor:fixedpoint} gives a unique hyperbolic fixed point $p_0 \in B_0$ and that $\hyp{B_0}^u(p_0)$ is a u-plaque $\al_0$ of $B_0$.

    Since $B_0 \fweakcovers B_1$, by \Cref{def:weakCov} there is a subplaque $\bt_0$ in the interior of $\al_0$ such that $f(\bt_0)$ is a u-plaque $\al_1$ of $B_1$. Applying the remaining weak covering relations inductively, for each $1 \le k \le l$ we obtain a u-plaque $\al_k$ of $B_k$ that is the image of a subplaque of $\al_{k-1}$. Since $\hypW^u(p_0)$ is $f$-invariant, $\al_k \subof \hypW^u(p_0)$ for each $1 \le k \le l$.

    Finally, since $B_l \fweakcovers \Fcal$, there is a subplaque $\bt_l$ in the interior of $\al_l$ such that $f(\bt_l) \in \Fcal$. Therefore $f^{l+1}(\gam)$, for some subcurve $\gam \subof \al_0 \subof \hypW^u(p_0)$, belongs to $\Fcal$, so
    the unstable manifold $\Wu(p_0)$ activates the blender $(\Lam, \Fcal)$.
\end{proof}

For our example system, we use \Cref{lemma:activate} to prove
that there is a hyperbolic fixed point $p_0$ with one-dimensional
unstable manifold and which activates the blender.

\subsection{Dense stable and unstable sets}
\label{sec:denseW}

For a set $s$ in the interior of a weak dynamical box $B,$
we call $s$ a \emph{good seed} if the following property holds:
if $N \subof M$ is an embedded $u$-dimensional disk
such that,
\begin{enumerate}
    \item $N$ intersects $s$ at a point $x_0,$
    \item
    $\del N$ does not intersect $B,$ and
    \item
    if $x \in N \cap B,$ then $T_x N \subof \Cone^u(x),$ \end{enumerate}
then the connected component $N \cap B$ through $x_0$ is a u-plaque for $B.$

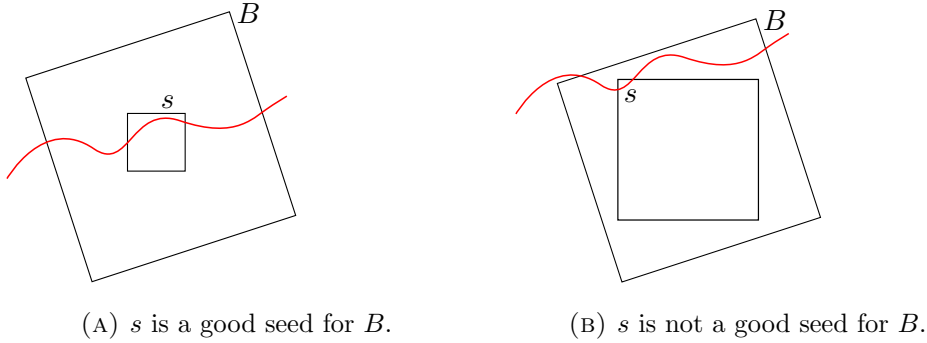
\begin{figure}[ht]
    \centering
    \begin{subfigure}{0.48\textwidth}        
        \def\svgwidth{4cm}
\begingroup%
  \makeatletter%
  \providecommand\color[2][]{%
    \errmessage{(Inkscape) Color is used for the text in Inkscape, but the package 'color.sty' is not loaded}%
    \renewcommand\color[2][]{}%
  }%
  \providecommand\transparent[1]{%
    \errmessage{(Inkscape) Transparency is used (non-zero) for the text in Inkscape, but the package 'transparent.sty' is not loaded}%
    \renewcommand\transparent[1]{}%
  }%
  \providecommand\rotatebox[2]{#2}%
  \newcommand*\fsize{\dimexpr\f@size pt\relax}%
  \newcommand*\lineheight[1]{\fontsize{\fsize}{#1\fsize}\selectfont}%
  \ifx\svgwidth\undefined%
    \setlength{\unitlength}{147.4015748bp}%
    \ifx\svgscale\undefined%
      \relax%
    \else%
      \setlength{\unitlength}{\unitlength * \real{\svgscale}}%
    \fi%
  \else%
    \setlength{\unitlength}{\svgwidth}%
  \fi%
  \global\let\svgwidth\undefined%
  \global\let\svgscale\undefined%
  \makeatother%
  \begin{picture}(1,0.96153846)%
    \lineheight{1}%
    \setlength\tabcolsep{0pt}%
    \put(0.52959582,0.6149986){\color[rgb]{0,0,0}\makebox(0,0)[lt]{\lineheight{1.25}\smash{\begin{tabular}[t]{l}$s$\end{tabular}}}}%
    \put(0.78159926,0.89466405){\color[rgb]{0,0,0}\makebox(0,0)[lt]{\lineheight{1.25}\smash{\begin{tabular}[t]{l}$B$\end{tabular}}}}%
    \put(0,0){\includegraphics[width=\unitlength,page=1]{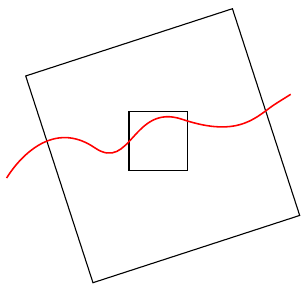}}%
  \end{picture}%
\endgroup%

        \caption{$s$ is a good seed for $B$.}
        \label{fig:goodseed}
    \end{subfigure}
    \hfill
    \begin{subfigure}{0.48\textwidth}        
        \def\svgwidth{4.2cm}
\begingroup%
  \makeatletter%
  \providecommand\color[2][]{%
    \errmessage{(Inkscape) Color is used for the text in Inkscape, but the package 'color.sty' is not loaded}%
    \renewcommand\color[2][]{}%
  }%
  \providecommand\transparent[1]{%
    \errmessage{(Inkscape) Transparency is used (non-zero) for the text in Inkscape, but the package 'transparent.sty' is not loaded}%
    \renewcommand\transparent[1]{}%
  }%
  \providecommand\rotatebox[2]{#2}%
  \newcommand*\fsize{\dimexpr\f@size pt\relax}%
  \newcommand*\lineheight[1]{\fontsize{\fsize}{#1\fsize}\selectfont}%
  \ifx\svgwidth\undefined%
    \setlength{\unitlength}{158.74015748bp}%
    \ifx\svgscale\undefined%
      \relax%
    \else%
      \setlength{\unitlength}{\unitlength * \real{\svgscale}}%
    \fi%
  \else%
    \setlength{\unitlength}{\svgwidth}%
  \fi%
  \global\let\svgwidth\undefined%
  \global\let\svgscale\undefined%
  \makeatother%
  \begin{picture}(1,0.89285714)%
    \lineheight{1}%
    \setlength\tabcolsep{0pt}%
    \put(0.35324856,0.60338562){\color[rgb]{0,0,0}\makebox(0,0)[lt]{\lineheight{1.25}\smash{\begin{tabular}[t]{l}$s$\end{tabular}}}}%
    \put(0.79191664,0.83075947){\color[rgb]{0,0,0}\makebox(0,0)[lt]{\lineheight{1.25}\smash{\begin{tabular}[t]{l}$B$\end{tabular}}}}%
    \put(0,0){\includegraphics[width=\unitlength,page=1]{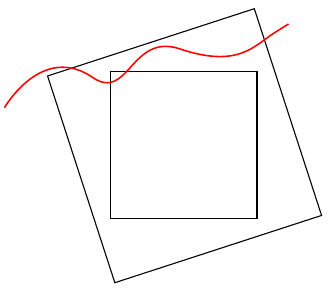}}%
  \end{picture}%
\endgroup%

        \caption{$s$ is not a good seed for $B$.}
        \label{fig:badseed}
    \end{subfigure}
    \caption{Definition of good seed.}
    \label{fig:seed}
\end{figure}

\begin{proposition} \label{prop:goodseeds}
    Let $\{s_i\}$ be a covering of $M$
    and let $\{B_i\}$ be a covering of $M$ as in \Cref{prop:weakph}
    such that each set $s_i$ is a good seed for $B_i.$
    Then,
    for any leaf $L$ of the unstable foliation of $f$
    and any subset $U \subof L$ with non-empty interior,
    there is $n \ge 1$ such that $f^n(U)$ contains a u-plaque
    of one of the boxes $B_i.$
\end{proposition}

\begin{proof}
    Let $x \in \mathrm{int}(U)$ and let $N_0 \subof U$ be a compact disc centered at $x$ in the leaf $L.$ Since $\{s_i\}$ covers $M,$ for each $n \ge 0$ choose an index $a_n$ with $f^n(x) \in s_{a_n}.$

    Each box $B_i$ has a $C^1$ projection $\pi^u : B_i \to [-1,1]^u$ onto its unstable coordinates. Since the unstable cones are uniformly expanded by \Cref{prop:weakph}, $Df$ uniformly expands the $\pi^u$-projection of vectors in $\Cone^u$ by a factor $\mu > 1$.

    Since $N_0$ is tangent to $\Cone^u$ and $f$ preserves the cone family, $f^n(N_0)$ is a $C^1$ disc tangent to $\Cone^u_{B_{a_n}}$ for every $n \ge 0,$ and the diameter of $\pi^u(f^n(N_0) \cap B_{a_n})$ grows at least as $C\mu^n$ for a constant $C > 0$ depending only on $N_0$ and the atlas. In particular, for all large enough $n,$ the image $f^n(N_0)$ overflows $B_{a_n}$ in such a way that
    some subset of $f^n(N_0)$ is a disc $N$ as in the definition of a
    good seed.
\end{proof}

As a consequence, we have the following criterion for the density of $\hypW^s(p_0)$.

\begin{proposition} \label{prop:densestable}
    Let $\{B_i\}$ be a finite family of weak dynamical boxes
    covering $M$ as in \Cref{prop:goodseeds}.
    Let $T_0$ be a strong dynamical box such that $T_0 \fcovers T_0$
    and let $p$ be the unique fixed point in $T_0.$
    Suppose for every $i,$ there is a finite sequence
    $X_1, \ldots, X_m$ of weak dynamical boxes such that
    \[
        B_i \fweakcovers X_1
        \fweakcovers \cdots \fweakcovers X_m
        \fweakcovers T_0.
    \]
    Then, $\hypW^s(p)$ is dense in $M$ and tangent to $\Ecs.$
\end{proposition}

\begin{proof}
    Let $U \subof M$ be open and non-empty, and $L$ be an unstable leaf intersecting $U$. Then the set $U \cap L$ has non-empty interior in $L.$ By \Cref{prop:goodseeds}, there exist $n \ge 1$ and an index $i$ such that $f^n(U)$ contains a u-plaque $\al$ of $B_i.$

    By hypothesis there is a chain $B_i \fweakcovers X_1 \fweakcovers \cdots \fweakcovers X_m \fweakcovers T_0.$ Applying the definition of the weak covering relation at each step, the u-plaque $\al$ contains a subplaque $\beta$ such that $f^{m+1}(\beta)$ is a u-plaque $\gam$ of $T_0,$
    and $\gam \subof f^{n+m+1}(U).$

    \Cref{cor:transv-inter} implies that $\gam$ intersects $\hyp{T_0}^s(p)$ in a point $y$. Thus $y \in f^{n+m+1}(U) \cap \hypW^s(p).$ Since $\hypW^s(p)$ is $f$-invariant, $f^{-(n+m+1)}(y) \in U \cap \hypW^s(p)$, and $\hypW^s(p)$ is dense in $M.$
\end{proof}

For our example system on the 3-torus,
we use \Cref{prop:weakph} and \Cref{prop:densestable}
to establish that $f$ has a weakly partially hyperbolic splitting
$\Ecs \oplus \Eu$ with $\dim(\Eu) = 1$
and that the fixed point $p_0$ which activates the blender
has a dense stable manifold tangent to $\Ecs.$

We also apply \Cref{prop:weakph} and \Cref{prop:densestable}
to $f \inv$ in place of $f$ and establish that
$f$ has a weakly partially hyperbolic splitting
$\Es \oplus \Ecu$ with $\dim(\Es) = 1$
and a hyperbolic fixed point $q_0$
whose two-dimensional unstable manifold is dense in $M$
and tangent to $\Ecu.$
We establish the existence of $q_0$
via a strong covering $B_q \fcovers B_q$ where $B_q$ is a subset of $V$
and it therefore follows that $q_0 \in \Lam.$

All of these objects together,
$p_0, q_0, \Lam, V, \Fcal,$ and the partially hyperbolic splitting of $f,$
establish that $f$ is robustly transitive by \Cref{prop:rt-bdv}.

\section{Implementation}
\label{sec:implementation}

For the implementation, we consider the derived-from-Anosov family of diffeomorphisms
$$f_{k,b}(x,y,z) = \psi_b \circ A_k(x, y, z) = (kx -y - z, x + y - b \sin (2 \pi x), x),$$
with $k \in \bbN$ and $b \in \bbR$. For this family, we can prove the following result.

\begin{proposition}
\label{prop:DA-non-anosov}
    If $b > \dfrac{1}{2\pi}$, then $f_{k,b}$ has at least three fixed points, namely $p_0 = (0, 0, 0)$, $q_0 = (x_0, (k-2)x_0, x_0)$ and $-q_0$, where $x_0$ is a positive root for $x - b \sin(2\pi x)$ belonging to $(0, \frac{1}{2})$. Additionally, these fixed points are hyperbolic, $Df_{k,b}(p_0)$ has unstable dimension equal to one, and $Df_{k,b}(q_0) = Df_{k,b}(-q_0)$ has unstable dimension equal to two.
\end{proposition}

We do not offer an analytical proof of the above properties because they are proved within our method to prove robust transitivity. Additionally, the strength of our method is that we do not need to know any of this information \textit{a priori} to apply it.

The proof of \Cref{teoA} is carried out by the C++ program whose structure is summarized in this section. The program uses the \texttt{CAPD} library \cite{CAPD} for interval arithmetic. Considering $p_0$ and $q_0$ the fixed points given by \Cref{prop:DA-non-anosov}, the main four verification routines are:
\begin{enumerate}[label=(\Roman*)]
    \item \texttt{verifyHorseshoe}: establishes the hyperbolic transitive set $\Lambda \subof \bbT^3$ using the strong covering relation with $u = 2$ (\Cref{sec:part1});
    \item \texttt{verifyActivation}: builds and validates the blender $(\Lambda, \mathcal{F})$ and verifies that $W^u(p_0)$ activates it (\Cref{sec:part2});
    \item \texttt{verifyUnstable}: verifies the weak partially hyperbolic splitting $E^{cs} \oplus E^u$ and that every unstable segment intersects $W^s(p_0)$ (\Cref{sec:part3,sec:part4});
    \item \texttt{verifyStable}: verifies the weak partially hyperbolic splitting $E^s \oplus E^{cu}$ and that every stable segment intersects $W^u(q_0)$ (\Cref{sec:part3,sec:part4}).
\end{enumerate}

We use interval arithmetic, which is a formalization of operations with intervals that represent inequalities. We use the notation $\bm{I} = [a, b]$, with intervals in bold. We say that $a$ is the \textit{lower bound} and $b$ is the \textit{upper bound} for $\bm{I}$.

Since the interval $\bm{I}$ represents all real values $t$ satisfying $a \leq t \leq b$, if we consider another interval $\bm{J} = [c, d]$, then the sum $\bm{I} + \bm{J}$ should represent all the possible values of $t + s$ where $s$ satisfies $c \leq s \leq d$. Thus, $\bm{I} + \bm{J} = [a + c, b + d]$. Other operations for intervals can be deducted similarly using inequalities, and we refer to \cite[Chapter 2]{warwickBook} for further details.

When implemented with a programming language, the interval resulting from any mathematical operation takes into account rounding errors on its lower and upper bounds, providing then rigorous inequalities for the values involved. Therefore, interval arithmetic allows us to perform several millions of inequalities rigorously, being then a powerful tool for computer-assisted proofs.

We represent the three-dimensional boxes as vectors of intervals, and we implement the function $f = f_{k,b}$ as an interval map. This means that $[f(\bm{I})]$, $\bm{I} = (\bm{I_1}, \bm{I_2}, \bm{I_3})$, is a product of intervals containing the actual image $f(\bm{I})$, and we call it an \textit{interval enclosure of $f(\bm{I})$}.

Section \ref{sec:hset} relates our definitions from \Cref{sec:dynBox} with similar concepts in the literature, as well describe the construction of dynamical boxes and a criterion to check for covering relation between them (\Cref{lem:checkCov}). The sections that follow (\Cref{sec:part1,sec:part2,sec:part3,sec:part4}) describe each one of the four parts of the implementation.

\subsection{H-sets, covering relation and cone preservation}
\label{sec:hset}

The notion of strong dynamical box introduced in \Cref{sec:dynBox} is similar to related notions in the literature. In \cite[Definition 1]{MZ, MZ2}, the authors use compact sets homeomorphic to the product $\overline{B_u(0,1)} \times \overline{B_s(0,1)}$ of closed unit balls in $\mathbb{R}^u$ and $\mathbb{R}^s$, respectively, to check local properties of the system. These sets are defined below.

\begin{definition}
    \label{def:h-set}
    If $M$ is a Riemannian manifold, we say that a compact subset $X \subseteq M$ is an \textit{h-set} if there is a homeomorphism
    $$\varphi: \overline{B_u(0,1)} \times \overline{B_s(0,1)} \to X,$$
    where $u + s = \dim M$.

    Additionally, we define the \textit{u-boundary} and the \textit{s-boundary} of $X$ as $X^u = \varphi(\cube^u)$ and $X^s = \varphi(\cube^s)$, where $\cube^u = \partial(\overline{B_u(0,1)}) \times \overline{B_s(0,1)}$ and $\cube^s = \overline{B_u(0,1)} \times \partial(\overline{B_s(0,1)})$.    
\end{definition}

Our dynamical boxes are particular cases of h-sets. We use the maximum norm on $\mathbb{R}^u$ and $\mathbb{R}^s$, and then each h-set is homeomorphic to the cube $\cube = [-1, 1]^{\dim M}$. Apart from being the most natural way to implement these sets using interval arithmetic, this allows us to make the distinction between the expanding and contracting dimensions on each box depending on the context, which is more suitable for partially hyperbolic systems.

To fit a dynamical box into the definition of an h-set, we need to fix the dimension of the unstable and stable directions. But the box is the same, regardless of this choice. The only difference is the definition of the $u$-boundary and $s$-boundary.

\begin{notation} 
    A \textit{center-unstable dynamical box}, or a \textit{cu-dynamical box} is a dynamical box viewed as an h-set with $u = 2$ and $s = 1$. Analogously, a \textit{center-stable dynamical box}, or a \textit{cs-dynamical box} is a dynamical box viewed as an h-set with $u = 1$ and $s = 2$.

    The ``center-unstable'' in the above notation means that, in local coordinates, $\bbR^u$ represents the direction of the center-unstable manifold, that does not need to present expanding behavior at every point.
\end{notation}

A topological way to see ``hyperbolicity'' when iterating h-sets is given by the following notion.

\begin{definition}[\cite{MZ, MZ2}]
    \label{def:covering}
    Consider two h-sets $X, Y \subseteq M$ that are homeomorphic to $\cube = [-1,1]^u \times [-1,1]^s$ via invertible maps 
    \begin{align*}
        \alpha: &\, \cube \subset B(0, c_X) \to U_X \supset X \mbox{ and}\\
        \beta: &\, \cube \subset B(0, c_Y) \to U_Y \supset Y,
    \end{align*}
    where the balls $B(0, c_X)$ and $B(0, c_Y)$ belong to $\bbR^{u + s}$. Let $X^u/Y^u$ and $X^s/Y^s$ be their u and s-boundaries.
    
    Given $f: M \to M$, the \textit{covering relation} occurs (and we say that \textit{$X$ $f$-covers $Y$}) if there is a homotopy $H: [0,1] \times \cube \to \bbR^3$ such that 
    \begin{itemize}
        \item $H(0, \cdot) = \floc = \beta^{-1} \circ f \circ \alpha$ ($f$ in local coordinates, where $\beta^{-1}$ is well defined over $f(X)$ if $X$ is sufficiently small);
        \item $H(t,\cube^u) \cap \cube = \varnothing$ for all $t \in [0,1]$;
        \item $H(t, \cube) \cap \cube^s = \varnothing$ for all $t \in [0,1]$;
        \item $H(1, (x_u, x_s)) = (Ax_u,0)$, where $A: \bbR^u \to \bbR^u$ is a linear map with
        $$A(\partial([-1,1]^u)) \subseteq \bbR^u \setminus [-1,1]^u.$$
    \end{itemize}
\end{definition}

Note that the above definition is adapted to any manifold $M$, while most authors only define it for $\bbR^n$. Note as well that this notion of covering relation, unlike the strong and weak versions we define in \Cref{sec:main-statements}, is only topological, it does not involve cones. This definition means that the intersection between $f(X)$ and $Y$ looks like the one in Figure \ref{fig:covering}.

\begin{figure}[ht]
   \centering
    \def\svgwidth{7cm}
\begingroup%
  \makeatletter%
  \providecommand\color[2][]{%
    \errmessage{(Inkscape) Color is used for the text in Inkscape, but the package 'color.sty' is not loaded}%
    \renewcommand\color[2][]{}%
  }%
  \providecommand\transparent[1]{%
    \errmessage{(Inkscape) Transparency is used (non-zero) for the text in Inkscape, but the package 'transparent.sty' is not loaded}%
    \renewcommand\transparent[1]{}%
  }%
  \providecommand\rotatebox[2]{#2}%
  \newcommand*\fsize{\dimexpr\f@size pt\relax}%
  \newcommand*\lineheight[1]{\fontsize{\fsize}{#1\fsize}\selectfont}%
  \ifx\svgwidth\undefined%
    \setlength{\unitlength}{226.77165354bp}%
    \ifx\svgscale\undefined%
      \relax%
    \else%
      \setlength{\unitlength}{\unitlength * \real{\svgscale}}%
    \fi%
  \else%
    \setlength{\unitlength}{\svgwidth}%
  \fi%
  \global\let\svgwidth\undefined%
  \global\let\svgscale\undefined%
  \makeatother%
  \begin{picture}(1,0.625)%
    \lineheight{1}%
    \setlength\tabcolsep{0pt}%
    \put(0,0){\includegraphics[width=\unitlength,page=1]{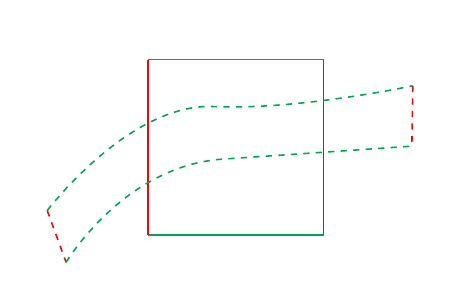}}%
    \put(0.43045468,0.50852667){\color[rgb]{0,0,0}\makebox(0,0)[lt]{\lineheight{1.25}\smash{\begin{tabular}[t]{l}$\textcolor{Green}{D^s}$\end{tabular}}}}%
    \put(0.68819599,0.17846994){\color[rgb]{0,0,0}\makebox(0,0)[lt]{\lineheight{1.25}\smash{\begin{tabular}[t]{l}$\textcolor{red}{D^u}$\end{tabular}}}}%
    \put(0.3169744,0.45045439){\color[rgb]{0,0,0}\makebox(0,0)[lt]{\lineheight{1.25}\smash{\begin{tabular}[t]{l}$D$\end{tabular}}}}%
    \put(0.70065587,0.35213372){\color[rgb]{0,0,0}\makebox(0,0)[lt]{\lineheight{1.25}\smash{\begin{tabular}[t]{l}$\floc(D)$\end{tabular}}}}%
  \end{picture}%
\endgroup%

    \caption{Covering relation in local coordinates.}
    \label{fig:covering}
\end{figure}
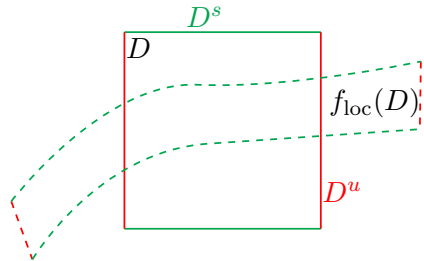

To have the full notion of hyperbolicity including the tangent space level, with contraction and expansion of vectors, we need not only the above covering relation, but a condition on cones. 

We recall that the \textit{dimension} of a cone is the dimension of the biggest subspace contained in it. If two h-sets with cone families have the same dimensions $u$ and $s$, we may ask if the cones are preserved under the derivative. In this sense, we have cone preservation as in \Cref{def:strongCov}.

The notion of u-plaque introduced in \Cref{def:u-plaque} is in general referred to as a \textit{horizontal disk} in \cite{MZ,MZ2,Zglic2009,CKOZ-2025}.

\begin{remark}
    \label{rem:covRel-implies-weakCov}
    Consider two weak dynamical boxes $X, Y$ as h-sets, and suppose their unstable cones $\Cone^u_X$ and $\Cone^u_Y$ are constant over $\cube = [-1, 1]^{u+s}$. If $(X, \Cone^u_X)$ and $(Y, \Cone^u_Y)$ satisfy the covering relation from \Cref{def:covering} with cone preservation given by
    \begin{equation}
        \label{eq:cone-preservation}
        [D\floc (\cube)](\Cone^u_X) \mbox{ is contained in the interior of } \Cone^u_Y,
    \end{equation}
    where $[D\floc (\cube)]$ is the interval enclosure of this derivative (which implies item (1) of \Cref{def:weakCov}), then we have a weak covering relation between the dynamical boxes $X$ and $Y$, meaning that we have item (2) of \Cref{def:weakCov}. This is a direct consequence of \cite[Lemma 30]{CKOZ-2025} and \cite[Theorem 7]{Zglic2009}, where the first takes the preservation of cones at tangent level given by \Cref{eq:cone-preservation} and brings it to the manifold level via the Mean Value Theorem, and the second uses the cone condition at the manifold level, together with the covering relation, to prove the preservation of u-plaques.
    
    For a proof of the strong covering relation using related notions in the literature, see \Cref{lem:covRel-quadCone-imply-strongCov}.
\end{remark}

\begin{remark}
    \label{rem:coding-references}
    In \cite[Lemma 9]{Zglic2009} the author proves that, given a finite sequence of covering relations (from \Cref{def:covering}) with a cone condition (at the manifold level) between $k+1$ h-sets that starts and ends on the same $(X_0, \Cone_{X_0})$, we have that the h-set $X_0$ contains a unique periodic point of period $k$ whose stable/unstable manifolds with respect to $f^k$ are given exactly by the points that stay in $X_0$ under forward/backward iterations of $f^k$. This is the first result in the literature, to the best of our knowledge, that implies the coding in \Cref{thm:coding}. But, for this coding to define a continuous semiconjugacy, one needs some expansiveness argument. That appears, for instance, in \cite[\S A.2]{CKOZ-2025}.
\end{remark}

\subsubsection{Cones}
\label{sec:cones}

In most places in the code, cones in $\bbR^3$ are generated by a rectangle of the form $(1, c_y, c_z)$ with $c_y, c_z$ being intervals,
so that a vector $(v_x, v_y, v_z)$ belongs to the cone if and only if $v_y/v_x \in c_y$ and $v_z/v_x \in c_z$.

\begin{figure}[ht]
   \centering
    \def\svgwidth{7cm}
    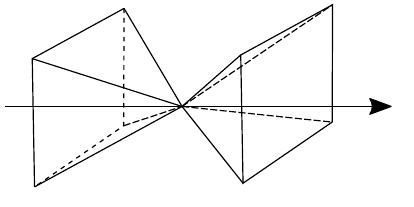
    \caption{Cone given by $\Cone(1) = \{1, [-1,1], [-1,1] \}$.}
    \label{fig:ucone}
\end{figure}

In local coordinates the expanding direction is always aligned with the first coordinate (even for the inverse, see \texttt{splitting} function, explained below, used to define the local coordinates). In particular, we use cones of the form 
\begin{equation} 
    \label{eq:cones}
        \Cone(a) = \{1, [-a, a], [-a, a] \}\\
\end{equation}
for $a \in (0, 1]$, see Figure \ref{fig:ucone}.

The cones given in \Cref{eq:cones} are used to prove partial hyperbolicity (\Cref{sec:part3}). They are also the cones used to verify the weak covering relation (see \Cref{rem:covRel-implies-weakCov}) for the activation of the blender (\Cref{sec:activation}) and for the dense manifold verifications (\Cref{sec:part4}).

If $Q$ is a quadratic form in $\bbR^n$, then it defines two cones with 

\begin{equation}
\label{eq:quadCones}
\begin{split}        
&\Cone^u = \{v \in \bbR^n \, : \, Q(v) \geq 0\}\\
&\Cone^s = \{v \in \bbR^n \, : \, Q(v) \leq 0\}.
\end{split}
\end{equation}

For the horseshoe verification (\Cref{sec:part1}), each seed box uses the cone associated with the quadratic form $Q = \operatorname{diag}\{1, 1, -1\}$. This is the cone we use to verify the strong covering relation (see \Cref{lem:covRel-quadCone-imply-strongCov}), which also appears in the verification of hyperbolic fixed points in \Cref{sec:activation} and in \Cref{sec:part4}.

\subsubsection{Constructing dynamical boxes}
\label{sec:coneBox}

Note that, if $f = f_{k,b}$ has its invariant splitting not ``aligned'' with the canonical basis of $\mathbb{R}^3$, then the interval enclosure $[f(\bm{I})]$ is bigger than the actual image of $\bm{I}$ under this map. To avoid the accumulation of overestimation (known as \textit{wrapping effect}), we use affine local coordinates ``aligned with the dynamics'' to define $\varphi$ for our dynamical boxes. More precisely, to represent computationally dynamical boxes and their good seeds, we use respectively \emph{cone boxes} and \emph{seed boxes}.

A \textit{cone box} is a pair $(\varphi, \Cone)$ where $\varphi = (A, p)$ is an affine map given by a matrix and a vector that define $\varphi : [-1,1]^3 \to \bbT^3$ as $\varphi(x) = Ax + p$, and $\Cone$ is a constant cone family defined in local coordinates as in \Cref{eq:cones}.

A \textit{seed box} is a cone box together with an axis-aligned box $\bm{s}$ (the \emph{seed}).
The seed represents the sets $\bm{s_i}$ from \Cref{prop:lam}, and the cone box is the larger dynamical box $B_i$ that covers it.

The affine map $\varphi$ is constructed by an approximation of the partially hyperbolic splitting computed by the \texttt{splitting} routine. By using the Singular Value Decomposition (SVD) for a matrix representing the derivative $Df^k$ at a point $p$,  we can find (an approximation of) the unit vector in $\bbR^3$ which is contracted the most by $Df^k$ at this point. This vector closely approximates the stable direction of the splitting at that point. Similarly, we can calculate the SVD for a matrix representing the derivative of $f^{-k}$ to find a unit vector approximating the unstable direction.

The SVD for $Df^k$ also gives a unit vector $v^u = v_1$ corresponding to the direction of greatest expansion. This is not necessarily a good approximation of the unstable direction, but it is a good approximation of the normal vector to the center-stable plane $E^{cs}(p)$. A similar column vector $v^s = \tilde{v_1}$ coming from the SVD of $Df^{-k}$ approximates the normal vector to the center-unstable plane $E^{cu}(p)$, and the cross product $v^c = v_1 \times \tilde v_1$ of these two vectors gives a close approximation to the center direction. In these calculations, we found that using $k=2$ gives a close approximation for all three directions of the partially hyperbolic splitting, and so we use $k = 2$ in the C++ code.

The columns of $A$ are then given by the vectors described above, $A = [v^u \mid v^c \mid v^s]$, scaled by a factor specified in \texttt{Settings} that defines the scale of the dynamical box.

We remark that this splitting does not need to be rigorous, its validation is part of the proof of partial hyperbolicity using cones in \Cref{sec:part3}.

Given two cone boxes, the program evaluates the \emph{local map} $\floc = A_2^{-1} \circ f \circ A_1$ between them (up to a lattice translation on the 3-torus) and encloses its derivative over a box using interval arithmetic. The use of local coordinates allows us to avoid the wrapping effect.

\subsubsection{Checking the covering relation}
\label{sec:checkCov}

As in \cite[Lemma 29]{CKOZ-2025}, we prove a criterion for the covering relation that can be easily implemented for a two-dimensional direction of expansion. In our setting, we could use the inverse and check for the covering relation with one-dimensional direction of expansion, but the weak contraction for the inverse at the center direction is harder to validate.

To formulate the hypotheses, given two dynamical boxes $B_1$ and $B_2$, we split the unstable boundary $\cube^u = \partial^u\cube = \partial([-1,1]^2) \times [-1,1]$ of $B_1$ in local coordinates into four components
    \begin{align*}
        U_l &= \{-1\} \times [-1,1] \times [-1,1],\\
        U_r &= \{1\} \times [-1,1] \times [-1,1],\\
        C_l &= [-1,1] \times \{-1\} \times [-1,1],\\
        C_r &= [-1,1] \times \{1\} \times [-1,1],
    \end{align*}
as in \Cref{fig:lem-checkCov}. Additionally, consider $\floc$ as $f$ in local coordinates, and $$Z = [-1, 1]^2 \times (\bbR \setminus [-1, 1]).$$

\begin{figure}[ht]
    \def\svgwidth{13cm}
    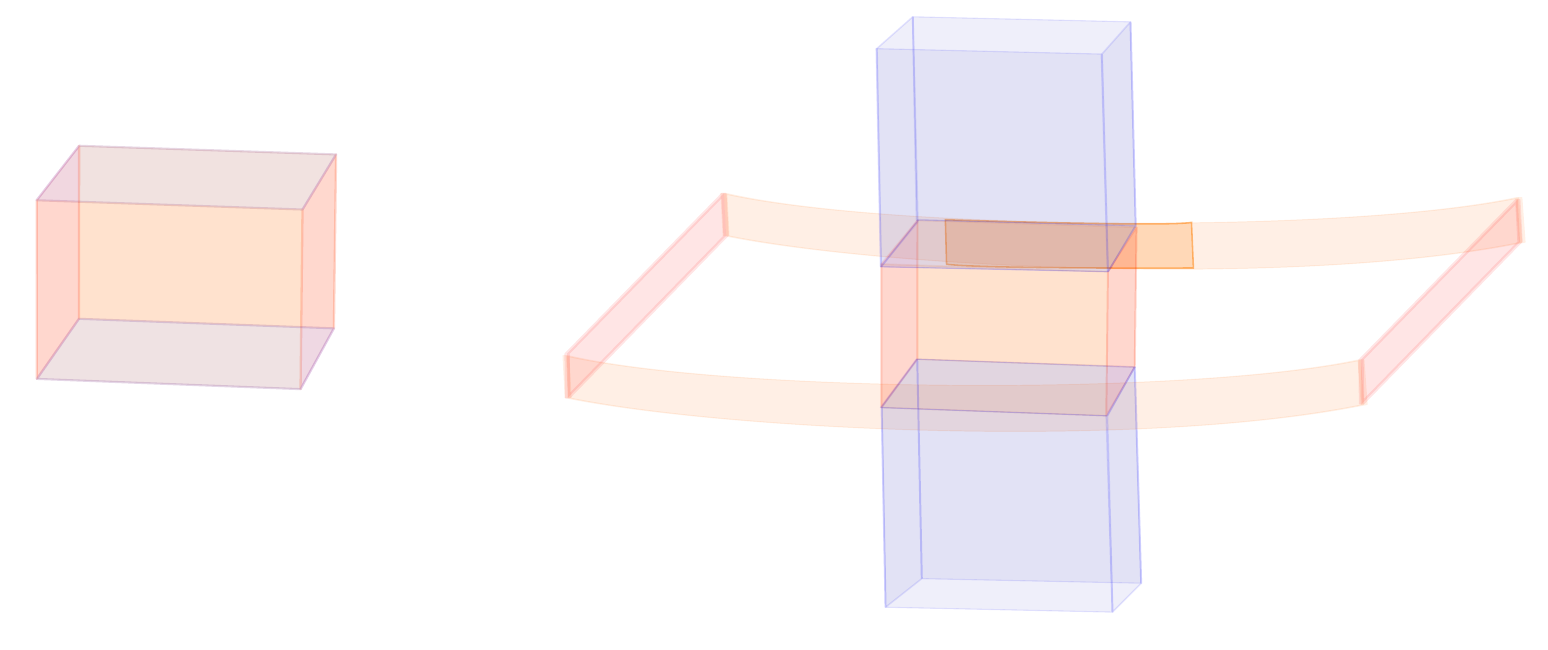
    \centering
    \caption{$\cube^u = \partial^u\cube = \partial([-1,1]^2) \times [-1,1] = \textcolor{red}{U_l} \cup \textcolor{red}{U_r} \cup \textcolor{orange}{C_l} \cup \textcolor{orange}{C_r}$ and $\textcolor{blue}{Z} = [-1, 1]^2 \times (\bbR \setminus [-1, 1])$. The conditions (2) and (3) in \Cref{lem:checkCov} are represented in the relative position of the images of the faces of $\cube^u$.}
    \label{fig:lem-checkCov}
\end{figure}

Recalling \Cref{def:covering}, we want conditions that are easily verifiable with inequalities and essentially imply that: the image of $\cube^u$ under $\floc$ is outside $\cube$, and the image of $\cube$ under $\floc$ does not intersect $\cube^s$. These conditions are expressed in the following result.

\begin{lemma}
    \label{lem:checkCov}
    Consider $B_1$ and $B_2$ center-unstable dynamical boxes for $f$ with $U_l$, $U_r$, $C_l$, $C_r$ and $Z$ defined as above. Consider also the projections $\pi_x: \bbR^3 \to \bbR$, $\pi_y: \bbR^3 \to \bbR$ to the first and second coordinates.
    
    Suppose that
    \begin{enumerate}
        \item $\floc(\cube) \cap Z = \varnothing$;
        \item \leavevmode\vspace{-\baselineskip} \begin{align*} 
            \pi_x(\floc(U_l)) < -1 &\quad \text{and} \quad\pi_x(\floc(U_r)) > 1 \mbox{ or}\\\pi_x(\floc(U_l)) > 1 &\quad \text{and} \quad \pi_x(\floc(U_r)) < -1;
        \end{align*}
    
        \item \leavevmode\vspace{-\baselineskip} \begin{align*}
            \pi_y(f_{C_l}) < -1 &\quad \text{and} \quad\pi_y(f_{C_r}) > 1 \mbox{ or}\\
            \pi_y(f_{C_l}) > 1 &\quad \text{and} \quad \pi_y(f_{C_r}) < -1,
        \end{align*}
        where $f_{C_l} = \floc(C_l) \cap ([-1,1] \times \mathbb{R}^2)$ and $f_{C_r} = \floc(C_r) \cap ([-1,1] \times \mathbb{R}^2)$.
    \end{enumerate}

    Then $B_1$ $f$-covers $B_2.$
\end{lemma}

\begin{proof}

    Given $(x, y, z) \in \cube$, we denote its image as $\floc(x, y, z) = (\overline{x}, \overline{y}, \overline{z})$. Condition (1) implies that we can define the following homotopy:
    
    \begin{alignat*}{3}
        &H_1: && \, [0,1] \times \cube &&\longrightarrow \bbR^3 \\
        & &&(t, (x, y, z)) &&\longmapsto (\pi_{(x,y)}(\overline{x}, \overline{y}, \overline{z}), (1-t) \pi_z(\overline{x}, \overline{y}, \overline{z})),
    \end{alignat*}
    where $\pi_{(x,y)}$ is the projection at the first two coordinates. This homotopy essentially ``squeezes'' the image of $\cube$ on the $z$ direction until it is centralized at $\overline{z} = 0$.

    After applying $H_1$, we claim that we only need $\pi_{(x,y)}(H_1(1, \cube^u))$ to contain $[-1,1]^2$ inside the region it bounds. Indeed, if this happens, then $H_1(1, \cube^u)$ is homotopic to $[-2, 2] \times \{0\}$ via the homotopy $H_2$ defined as 
    
    \begin{alignat*}{3}
        &H_2: && \, [0,1] \times \cube &&\longrightarrow \bbR^3 \\
        & &&(t, (x, y, z)) &&\longmapsto (t A(x, y) + (1-t) \pi_{(x,y)}(\overline{x}, \overline{y}, \overline{z})), 0),
    \end{alignat*}
    where $A = \operatorname{diag}\{2, 2\}$, and $\pi_{(x,y)}(H_2(t, \cube^u))$ contains $[-1,1]^2$ for all $t \in [0,1]$. Then, the composition of $H_1$ and $H_2$ is a homotopy $H$ satisfying the definition of covering relation.
    
    Now, let us prove that conditions (2), and (3) from the lemma imply that $\pi_{(x,y)}(H_1(1, \cube^u))$ contains $[-1,1]^2$ inside the region it bounds. This region is homotopic to a circle. It has a bound on its perimeter that depends only on $f$ and it bounds a compact disk in $\bbR^2$, since $\floc(\cube))$ is diffeomorphic to a cube. 
    
    Note that $\cube^u = \partial^u\cube = \partial([-1,1]^2) \times [-1,1]$ is the union of $U_l$, $U_r$, $C_l$ and $C_r$, then $\pi_{(x,y)}(H_1(1, \cube^u))$ is a union of curves $\gamma^u_l = \pi_{(x,y)}(H_1(1, U_l))$, $\gamma^u_r = \pi_{(x,y)}(H_1(1, U_r))$, $\gamma^c_l = \pi_{(x,y)}(H_1(1, C_l))$, $\gamma^c_r = \pi_{(x,y)}(H_1(1, C_r))$.

    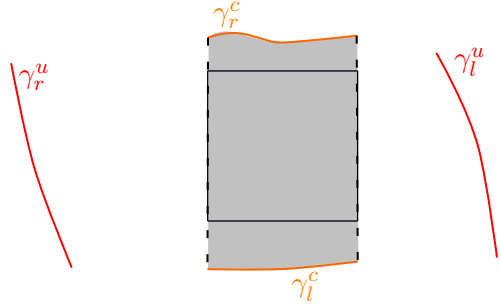
\begin{figure}[ht]
        \def\svgwidth{8cm}
\begingroup%
  \makeatletter%
  \providecommand\color[2][]{%
    \errmessage{(Inkscape) Color is used for the text in Inkscape, but the package 'color.sty' is not loaded}%
    \renewcommand\color[2][]{}%
  }%
  \providecommand\transparent[1]{%
    \errmessage{(Inkscape) Transparency is used (non-zero) for the text in Inkscape, but the package 'transparent.sty' is not loaded}%
    \renewcommand\transparent[1]{}%
  }%
  \providecommand\rotatebox[2]{#2}%
  \newcommand*\fsize{\dimexpr\f@size pt\relax}%
  \newcommand*\lineheight[1]{\fontsize{\fsize}{#1\fsize}\selectfont}%
  \ifx\svgwidth\undefined%
    \setlength{\unitlength}{226.77165354bp}%
    \ifx\svgscale\undefined%
      \relax%
    \else%
      \setlength{\unitlength}{\unitlength * \real{\svgscale}}%
    \fi%
  \else%
    \setlength{\unitlength}{\svgwidth}%
  \fi%
  \global\let\svgwidth\undefined%
  \global\let\svgscale\undefined%
  \makeatother%
  \begin{picture}(1,0.625)%
    \lineheight{1}%
    \setlength\tabcolsep{0pt}%
    \put(0,0){\includegraphics[width=\unitlength,page=1]{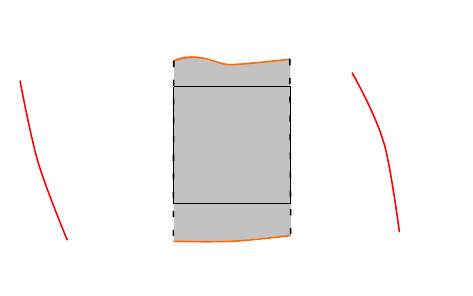}}%
    \put(0.37640479,0.52345749){\makebox(0,0)[lt]{\lineheight{1.25}\smash{\begin{tabular}[t]{l}$\textcolor{orange}{\gamma^c_r}$\end{tabular}}}}%
    \put(0.50563445,0.07635412){\makebox(0,0)[lt]{\lineheight{1.25}\smash{\begin{tabular}[t]{l}$\textcolor{orange}{\gamma^c_l}$\end{tabular}}}}%
    \put(0.77532158,0.44707748){\makebox(0,0)[lt]{\lineheight{1.25}\smash{\begin{tabular}[t]{l}$\textcolor{red}{\gamma^u_l}$\end{tabular}}}}%
    \put(0.05483266,0.42126596){\makebox(0,0)[lt]{\lineheight{1.25}\smash{\begin{tabular}[t]{l}$\textcolor{red}{\gamma^u_r}$\end{tabular}}}}%
  \end{picture}%
\endgroup%

       \centering
       \caption{$\pi_{(x,y)}(H_1(1, \cube^u))$ is homotopic to a circle containing these four curves and that do not intersect the region $G$ (in gray).}
       \label{fig:lem-cov}
    \end{figure}
    
    Condition (2) implies that, among $\gamma^u_l$ and $\gamma^u_r$, one of them lies to the right of $[-1, 1]^2$ and the other lies to the left. Condition (3) implies that $\gamma^c_l$ and $\gamma^c_r$ do not intersect $[-1,1]^2$. Taking into consideration that all these curves have bounded length, they are homotopic to \Cref{fig:lem-cov}, with a region $G$ containing $[-1,1]^2$ that the circle does not intersect (since $\gamma^c_{l/r}$ can only intersect $[-1,1] \times (1, \infty)$ or $[-1,1] \times (-\infty, 1)$ a finite number of times), and it only remains to connect these curves to form a circle.
    
    Since this circle is homotopic to $\pi_{(x,y)}(H_1(1, \cube^u))$, $\gamma^u_l$ and $\gamma^u_r$ each should have one end connected to $\gamma^c_l$ and one connected to $\gamma^c_r$. It is easy to see that, if we avoid the region $G$, there are only four ways to do that, and they all form circles containing $[-1,1]^2$.
\end{proof}

\begin{remark}
    The above lemma can be applied in dimension $n$ for a pair of h-sets with $u = 2$ and $s = n-2$.
\end{remark}

To check condition (1) for \Cref{lem:checkCov}, we can simply check that the whole image of $\cube$ has its $z$-coordinate inside $[-1 , 1]$, by checking $\pi_z(\floc(\cube)) \subseteq [-1, 1]$. A simple check condition for condition (3), similarly to condition (2), is to check that
\begin{align*}
    \pi_y(\floc(C_l)) < -1 &\quad \text{and} \quad \pi_y(\floc(C_r)) > 1 \mbox{ or}\\
    \pi_y(\floc(C_l)) > 1 &\quad \text{and} \quad \pi_y(\floc(C_r)) < -1.
\end{align*}
However, since we have weak expansion in the center direction, this may fail to happen even if a covering relation holds.

The function that checks the conditions of \Cref{lem:checkCov} in our implementation is \texttt{surfaceCovering}, while the covering relation with $u = 1$ is checked in the \texttt{curveCovering} function, both described briefly below.

For \texttt{curveCovering}, we need the projection of the image of the $x$-coordinate to strictly contain $[-1, 1]$, which is checked with Condition (2), while the $y$ and $z$-coordinates should only be checked to be inside this interval for points with $x$-coordinate in $[-1, 1]$ (see \cite[Lemma 29]{CKOZ-2025}). We then restrict to a subset of $\cube$ by computing a ``\textit{subcube}'', with the function \texttt{restrictToPreimage}, that is an enclosure $\bm{E}$ of all points $v \in \cube$ such that $\pi_x(\floc(v)) \in [-1, 1]$. We solve a linearized equation to check that the $y$ and $z$-coordinates of the image of $\bm{E}$ in local coordinates are both inside $[-1, 1]$.

For \texttt{surfaceCovering}, the check for Conditions (2) and (3) is translated into checking that points in $\cube^u$ are mapped outside $\cube$ as follows: we restrict to the subcube $\bm{E}$ as before, then if $\pi_x(\bm{E}) \subsetneq [-1, 1]$, the $x$-coordinate is mapped outside $\cube$. For points $v \in \bm{E}$ with $\vert \pi_y(v) \vert = 1$, we solve a linearized equation to check that the $x$ and $y$-coordinates of the image of $v$ in local coordinates are both outside $[-1, 1]$.
Then we verify that $\pi_z(\floc(\bm{E})) \subseteq [-1, 1]$ as in the \texttt{curveCovering} function. Further details are given as comments in the C++ code.

\subsection{Hyperbolicity and transitivity}
\label{sec:part1}

We need to find $\Lambda$ the transitive hyperbolic set to be the candidate for the blender. Before we describe the implementation, we introduce a result by \cite{Wilczak2010} to check if an invariant set is uniformly hyperbolic.

The notion of \textit{strong hyperbolicity} that follows implies, for cones defined using a particular quadratic form, cone preservation (\cite[Lemma 2.4]{Wilczak2010}) and expansion and contraction of vectors on unstable cones and stable cones respectively (\cite[Lemma 2.7, Lemma 2.8]{Wilczak2010}). 

Consider the matrix
$$Q = \operatorname{diag}\{ \underbrace{1, \cdots, 1}_{u}, \underbrace{-1, \cdots, -1}_{s} \},$$
where $u + s = n$. It defines the quadratic form in $\bbR^n$ given by
$$Q(v) = v^TQv = \Vert v_u \Vert^2 - \Vert v_s \Vert^2,$$
where $v = (v_u, v_s)$, and the cones given by this quadratic form are as in \Cref{eq:quadCones}.

\begin{definition}
    \label{def:strongHyp}
    We say that $f: M \to M$ is \textit{strongly hyperbolic} on $U \subseteq M$ with dimensions $u$ and $s$ (where $u+s = n = \dim M$), if $U =  \cup_i X_i$, with the family of h-sets $\mathcal{X} = \{(X_i, \varphi_i)\}_{i \in \mathcal{I}}$ satisfying:
    \begin{enumerate}
        \item $\varphi_i: B^u(0,1) \times B^s(0,1) \to B_i$ is an affine map, with $\varphi_i(x) = A_i(x) + p_i$, and
        \item the matrix
        \begin{equation}
            \label{eq:def-pos}
            [A_j Df(p) A_i^{-1}]^TQ[A_j Df(p) A_i^{-1}] - Q
        \end{equation}
        is positive definite for all $i, j$ such that $f(X_i) \cap X_j \neq \varnothing$ and all $p \in X_i$.
    \end{enumerate}
\end{definition}

The condition given by the matrix in \Cref{eq:def-pos} being positive definite implies that the cone $\Cone^u$ is preserved under the derivative of $f$ and the cone $\Cone^s$ is preserved under the derivative of $f^{-1}$, both in local coordinates. In fact, this is true even for a weaker condition. If, instead, we have that there is $m > 0$ such that 
        \begin{equation}
            \label{eq:def-pos-m}
            [A_j Df(p) A_i^{-1}]^TQ[A_j Df(p) A_i^{-1}] - mQ
        \end{equation}
is positive definite for all $i, j$ such that $f(X_i) \cap X_j \neq \varnothing$ and all $p \in X_i$, then we have the same cone preservation (see the proof for $m = 1$ in \cite[Lemma 2.4]{Wilczak2010}, the general case being proved the same way). For the expansion and contraction of vectors inside these cones, however, we need $m=1$.

The hyperbolicity is simply an application of the following result, stated here for our setting (see \Cref{rem:quotient-problem}).

\begin{theorem}{\cite[Theorem 2.3]{Wilczak2010}}
    \label{teo:strongHyp}
    If $f: M \to M$ is strongly hyperbolic on $U$, then it is uniformly hyperbolic on $\Lambda = \cap_{n \in \mathbb{Z}} f^n(U)$.
\end{theorem}

In particular, the cone preservation with the quadratic form expressed in the strong hyperbolicity condition allows us to prove the strong covering relation as in \Cref{def:strongCov}.

\begin{lemma}
    \label{lem:covRel-quadCone-imply-strongCov}
    If a pair of h-sets $X$ and $Y$ satisfy $X$ $f$-covers $Y$ and $X \cup Y$ is strongly hyperbolic, then the strong covering relation $X \fcovers Y$ holds for the cones defined by the quadratic form $Q$.
\end{lemma}

\begin{proof}
    We regard $X$ and $Y$ as strong dynamical boxes (\Cref{def:strongDynBox}) with cones given by \Cref{eq:quadCones}. Let $F_p = A_Y Df(p) A_X^{-1}$ denote the local derivative for a point $p \in X$, where $\cube = [-1,1]^{u+s}$, $\varphi_X: \cube \to X$ and $\varphi_Y: \cube \to Y$ are affine coordinates for $X$ and $Y$, respectively, with $\varphi_X(q) = A_X(q) + p_X$ and $\varphi_Y(q) = A_Y(q) + p_Y$.
    
    We have by strong hyperbolicity that $F_p^T Q F_p - Q$ is positive definite for all $p \in X$. Then the cone preservation in the definition of strong covering relation follows from \cite[Lemma 2.4]{Wilczak2010}.

    By \cite[Lemma 8]{Zglic2009}, the strong hyperbolicity also implies that
    \begin{equation}
        \label{eq:zglic-cone}
        Q(\floc(x_1) - \floc(x_2)) > Q(x_1 - x_2)
        \quad \text{for all } x_1 \neq x_2 \in R,
    \end{equation}
    which is the cone condition in the sense of \cite{Zglic2009}.
    
    Thus, by \cite[Theorem 7]{Zglic2009}, for every u-plaque $\al$ of $X$ there is a sub-plaque $\bt$ with image a u-plaque of $Y$.

    Note as well that, if $\al$ is a u-plaque for $X$ with $\Cone^u = \{v \in \bbR^n \, : \, Q(v) \geq 0\}$, then $\al$ is the graph in local coordinates of a Lipschitz map with Lipschitz constant less than or equal to $1$. In particular, two points $q_1, q_2$ in $\varphi_X \inv (\al)$ satisfy $Q(q_1 - q_2) \geq 0$ by the definition of $Q$.
    
    Then uniqueness is a consequence of \Cref{eq:zglic-cone}. Indeed, suppose by contradiction, that the subplaque $\bt$ of $\al$ is not unique. Since $f$ is a diffeomorphism, $\floc$ is injective, which means that it has two different subplaques $\bt_1$ and $\bt_2$ being mapped to two different u-plaques $\floc(\bt_1)$ and $\floc(\bt_2)$. Consider $p_1 = (p^u, p_1^s) \in \floc(\bt_1)$ and $p_2 = (p^u, p_2^s) \in \floc(\bt_2)$ two points with the same first coordinate in $[-1, 1]^u$. By the definition of $Q$, $Q(p_1 - p_2) < 0$. Consider $q_1 = \floc \inv(p_1)$ and $q_2 = \floc \inv(p_2)$ the pre images of $p_1$ and $p_2$, which are points in $\al$. \Cref{eq:zglic-cone} implies that $Q(q_1 - q_2) < Q(p_1 - p_2) < 0$, which contradicts the fact that $\al$ is a u-plaque.

    The argument for s-plaques is analogous.
\end{proof}

\begin{remark}
    \label{rem:quotient-problem}
    The previous results in the literature, in particular from \cite{Wilczak2010} and \cite{Zglic2009}, that we use in the above proof, are for $\bbR^n$, not for general manifolds. This implies that they hold only locally for $M$, meaning that $X$ and $Y$ should be sufficiently small. For the torus case we deal with in this work, small enough means that $f(X)$ should be contained in a fundamental domain of the torus, or equivalently, that all coordinates of $f(X)$ should have diameter less than $1$.
\end{remark}

For the particular family we apply our results to, we want to find an invariant, hyperbolic and transitive subset of $\bbT^3$ that has expanding center direction for $f$. To look for that, we first look for a region in which the central direction is expanded under $Df$, which can be checked numerically using the splitting given in \Cref{sec:coneBox}. We also take this region to contain the fixed point $q_0$, that has its center direction being expanded. The \texttt{verifyHorseshoe} routine establishes the uniformly hyperbolic transitive set $\Lambda$ constructed using the blender region $V = V_x \ti S^1 \ti V_z$ and \Cref{prop:lam}.

We fix $V_x = [0.38, 0.48]$ and $V_z$ slightly larger, as discussed in \Cref{sec:blender-crit}. The region $V$ is divided into approximately cubic seed boxes using the function \texttt{buildHorseBoxes}.
The $x$ and $z$ directions are each split into $N_{\mathrm{hyp}} = 40$ equal subintervals, and the $y$
direction is split into $\lfloor N_{\mathrm{hyp}} / |V_x| \rfloor = 400$ subintervals to make the boxes roughly cubic. We then keep only the boxes $\bm{s_i}$ intersecting $V$ after applying $f$ and $f^{-1}$, see \Cref{fig:Utilde}. The hyperbolic set $\Lam = h(\Sig_{T})$ is given by the symbolic dynamics over the remaining boxes by \Cref{prop:lam}.

As we mentioned in \Cref{sec:hyp-trans}, we use the axis-aligned boxes $\bm{s_i}$ to define the symbolic dynamics and larger dynamical boxes $B_i$ to check for the covering relation (see \Cref{sec:checkCov}). Then for each remaining axis-aligned box $\bm{s_i}$, a cu-dynamical box (cone box with $u = 2$, $s = 1$) having $\bm{s_i}$ as a seed is constructed using \texttt{buildSeedBox} with scale $0.02$ and the \texttt{standardCone}. The scale is chosen to ensure that, every time the image of a seed box $\bm{s_i}$ meets another box $\bm{s_j}$, the corresponding dynamical boxes satisfy $B_i \fcovers B_j$, see \Cref{fig:bb}.

\begin{figure}[ht]
    \centering
    \begin{subfigure}{0.48\textwidth}
        \includegraphics[width=1.1\linewidth]{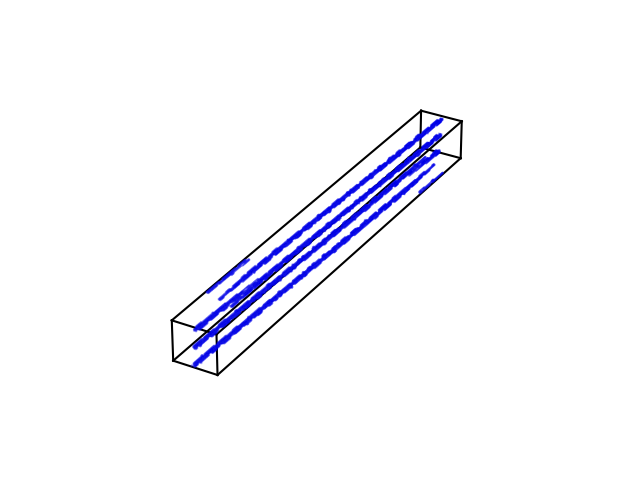}
        \caption{We plot (in blue) the boxes $\bm{s_i}$ remaining in $V$ after applying $f$ and $f^{-1}$.}
        \label{fig:Utilde}
    \end{subfigure}
    \hfill
    \begin{subfigure}{0.48\textwidth}
        \includegraphics[width=1.1\linewidth]{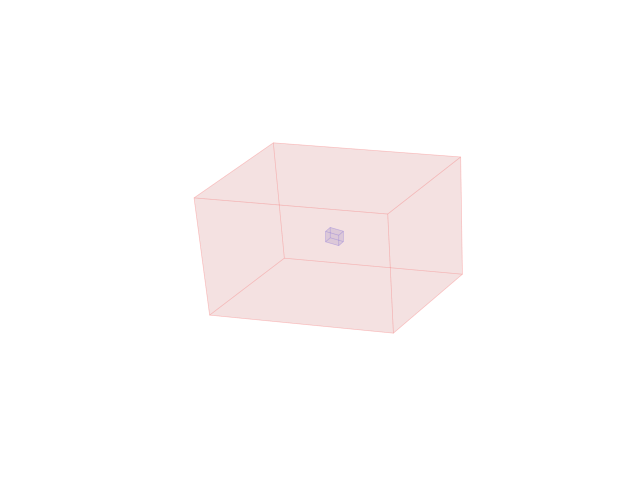}
        \caption{The scale of the bigger dynamical boxes (in red).}
        \label{fig:bb}
    \end{subfigure}
    \caption{Boxes used to prove the existence of a transitive hyperbolic set.}
    \label{fig:lambda}
\end{figure}

We then define the region $U = \cup_{i \in \mathcal{I}} B_i$ as the one given by the dynamical boxes $B_i$ covering each of the boxes $\bm{s_i}$. Applying \Cref{teo:strongHyp} to $U$ gives us that any invariant set contained in $U$ is uniformly hyperbolic.

Using interval arithmetic, we test if $f(\bm{s_i})$ intersects $\bm{s_j}$ as subsets of $\bbT^3$. If the computer fails to discard the possibility that $f(\bm{s_i}) \cap \bm{s_j} = \varnothing$, we verify if $B_i \fcovers B_j$. To do so, we use \Cref{lem:covRel-quadCone-imply-strongCov} (implemented in the function \texttt{strongCovering} with \Cref{lem:checkCov} and \Cref{eq:def-pos}). Then we add an edge from $\bm{s_i}$ to $\bm{s_j}$ to the directed graph $G$ representing the symbolic dynamics of $f$ restricted to $V$. If the covering relation fails, we stop the program and adjust the parameters, in particular the scale of the dynamical box.

If the graph is strongly connected, or if it only has one \textit{non-trivial connected component}, then the corresponding invariant set is transitive by \Cref{prop:lam}. In our setting, a \textit{trivial connected component} is formed by one vertex that does not contain self-loops and is not connected for the past and for future to other vertices. Such a vertex cannot be part of a bi-infinite sequence in the symbolic dynamics, and so can safely be discarded.

Describing the implementation more precisely, for each pair $(B_i, B_j)$, the \texttt{strongCovering} routine checks:
\begin{itemize}
    \item \texttt{checkIntersection}: if $$\floc(\varphi_i^{-1}(\bm{s_i})) \cap \varphi_j^{-1}(\bm{s_j}) \neq \varnothing,$$
    where $\floc = \varphi_j^{-1} \circ f \circ \varphi_i$. This is a more refined check for intersection of the image of the seed $\bm{s_i}$ with the seed $\bm{s_j}$ in local coordinates. If this fails, there is nothing to be checked.
    \item \texttt{surfaceCovering}: the strong covering relation with $u = 2$ as described in \Cref{sec:checkCov}.
    \item \texttt{quadConeTest}: the matrix 
    $$[A_j Df(p) A_i^{-1}]^T Q [A_j Df(p) A_i^{-1}] - Q$$ 
    is positive definite, where $Q = \operatorname{diag}\{1,1,-1\}$. This implies cone preservation in the sense of \Cref{def:strongHyp} and, by \Cref{teo:strongHyp}, uniform hyperbolicity of $\Lambda$.
\end{itemize}

The \texttt{buildIncidenceMatrix} routine loops over all pairs $(B_i, B_j)$. If the test \texttt{strongCovering} returns true, an edge $i \to j$ is added to the directed graph $G$.
Tarjan's algorithm is then applied to $G$ to verify that it is strongly connected, which by \Cref{prop:lam} implies that $\Lam = h(\Sig_{T})$ is transitive and that if a point $p$ has its forward orbit in $V$, then $p \in W^s(\Lam)$. Also, by construction, if $p$ has its full orbit in $V$, then $p \in \Lam$, and $\Lam$ is an invariant set contained in $U = \cup_{i \in \mathcal{I}} B_i$, so it is uniformly hyperbolic.

\subsection{Blender}
\label{sec:part2}

The main function of this part is \texttt{verifyActivation}. It builds the blender with the \texttt{buildBlender} routine as a union of ``bunches'', and the routine \texttt{blenderInvariant} looks, for each bunch in the blender, if there is a good ``branch'' that makes it invariant.

The bunches represent a family of curves that form a robust covering collection (\Cref{def:rcc}). Their construction in the
\texttt{buildBlender} routine is detailed in the following.

Remember we denote the region
containing the blender as $V = V_x \ti V_y \ti V_z \subof \bbR^3.$
Define $q_x$ as the midpoint of $V_x$
and define a set $P = q_x \ti V_y \ti V_z.$
Define an interval $J$ centered at $0$ such that $q_x + J = V_x.$

\begin{definition}
    \label{def:ubunch}
    
    Given a vector $u = (1, u_y, u_z)$ and $\delta > 0,$ we define a cone
    $\Cone(u, \delta) \subof \bbR^3$
    as all rescalings of vectors in
    \[
        m = m(u, \delta) :=
        \{1\} \ti (u_y + [-\delta, \delta]) \ti (u_z + [-\delta, \delta]).
    \]
    That is, $v \in \Cone(u, \delta)$ if $v = s \cdot w$
    for some $s \in \bbR$ and $w \in m.$
    
    With this cone, given a rectangle of the form $r = q_x \ti r_y \ti r_z,$
    we define a \emph{bunch} $\Fcal(r, u, \delta)$
    as the set of all $C^1$ curves $\gam : J \to \bbR^3$
    of the form
    \[
        \gam(t) = (q_x + t, \gam_y(t), \gam_z(t))
    \]
    where $\gam(0) \in r$ and $\gam'(t) \in \Cone(u, \delta)$
    for all $t \in J.$
\end{definition}

\begin{figure}[ht]
    \def\svgwidth{6cm}
    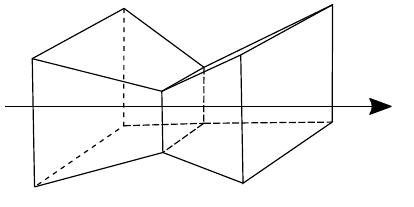
   \centering
   \caption{Each bunch in $\mathcal{F}$ is given by a rectangle, a direction $v^u$ and an ``opening'' $\delta$ that we allow along this direction.}
   \label{fig:ucurve}
\end{figure}

Note that, for a subinterval $J_0 \subof J$,
the set $\{ \gam(t) : \gam \in \Fcal(r, u, \delta), t \in J_0 \}$
is contained in the axis-aligned box
$$r + J_0 \cdot m
    = \{ p + t \cdot v \; : \; p \in r, \; t \in J_0, \; v \in m \}$$
where $m$ is defined as above.
In particular, every curve of the bunch has its image in $r + J \cdot m.$

We represent the blender as a finite union
of such bunches $\Fcal = \bigcup_i \Fcal_i.$
For this, we partition the set $P = q_x \ti V_y \ti V_z$
into a union of rectangles $r_i = q_x \ti V_{y,i} \ti V_{z,i}.$
For each rectangle, we calculate an approximation
$u_i = (1, u_{y,i}, u_{z,i})$ to the unstable direction $\Eu$
of the partially hyperbolic splitting at the midpoint of $r_i,$
and $\delta > 0$ is a fixed constant independent of $i.$
Then the bunch associated to $r_i$ is $\Fcal_i = \Fcal(r_i, u_i, \delta)$,
obtained by the \texttt{buildBunch} routine. 

The robust covering collection is defined by the bunches.
To verify the hypotheses of \Cref{prop:blender} in the \texttt{blenderInvariant} function,
for each bunch $\Fcal_i$ in the robust covering collection, our program checks that every curve in it has a subcurve whose image is in the robust covering collection.
To do that, the code computes all valid \emph{branches}: a list of integers $n_x$ such that, for each $n_x$, there is a restriction to a subinterval $J_i \subof J$ such that $\pi_x(f \circ \gam|_{J_i})$ crosses $V_x + n_x$ for every curve $\gam \in \Fcal_i$.

A branch is \emph{good} if the derivative maps the cone of the branch into the cone of every bunch whose rectangle intersects the image of the branch at the $q_x$-plane. Note that if a curve contained in $V$ crosses it from the left face to the right face, as described below, then it intersects at least one rectangle $r_i$, so the check for compatible cones is not void.
The bunch $\Fcal_i$ passes the invariance test if it has at least one good branch.

During the computation of valid branches, we need to check that the image of the bunch under an integer translation intersects the set $V$ in an appropriate way to allow for invariance. This is done using the \texttt{crossesV} function, that we explain as follows. For convenience, let $f_x, f_y, f_z : \bbR^3 \to \bbR$ denote the coordinates of the composition of the lifted map $f : \bbR^3 \to \bbR^3$
with the translation by $-k \in \bbZ^3$.
That is,
\[
    (f_x(p), f_y(p), f_z(p)) = f(p) - k.
\]

To check that a branch is valid, we write the interval $J_i$ as $J_i = [a_i, b_i],$ and recall that
$\gam(t) \in r_i + t m$ for all $\gam \in \Fcal_i$, $t \in J$ and $m$ defined as above.
The program computes enclosures of both
$f_x(r_i + a_i \cdot m)$ and $f_x(r_i + b_i \cdot m)$
and verifies that these lie in two distinct connected components of $\bbR \sans V_x$ in the function \texttt{crossesV}.
This guarantees by the Intermediate Value Theorem that
\[
    V_x \subof f_x(\gam(J_i))
\]
for any $\gam \in \Fcal_i.$
We also verify at this stage that $f_x(r_i + J_i \cdot m)$ does not intersect
any other integer translates of $V_x$ and that
\[
    f_z(r_i + J_i \cdot m) \subof V_z.
\]

For our example system in $\bbT^3$, we have $V_y = [0, 1]$, so there is no need to verify $f_y(r_i + J_i \cdot m) \subof V_y$ in the \texttt{crossesV} function.

Once we verify that the branch is valid, we next determine which rectangles
$r_j$ can intersect a curve of the form $f \circ \gam|_{J_i}$, with $\gam \in \Fcal_i$, with the \texttt{intersect} routine.
Since all of the rectangles $r_j$ are of the form
$r_j = q_x \ti r_{y,j} \ti r_{z,j},$
we first compute a small interval $T_i$ with the property that
$t \in J_i$ and $f_x \circ \gam(t) = q_x$
together imply that
$t \in T_i.$

We then compute an enclosure of the set
\[
    f \circ \gam(T_i) \subof f(r_i + T_i \cdot m_i).
\]
If $r_j$ is disjoint from $f(r_i + T_i m_i),$
then no curve of the form $f \circ \gam|_{J_i}$
with $\gam \in \Fcal_i$ belongs to $\Fcal_j$.
For those $r_j$ where an intersection is possible,
we verify that $Df$ maps the unstable cone $\Cone^u_i$
inside the unstable cone $\Cone^u_j$ using the
function \texttt{compatibleCones}. 

We verify all the conditions in the function \texttt{blenderInvariant} using strict inequalities, and so if a bunch is good for $f$, it is also good for any diffeomorphism $g: \bbT^3 \to \bbT^3$ which is $C^1$ close to $f$.

This finishes the verification that the family of curves $\Fcal$ represented by the bunches forms a robust covering collection. Together with the hyperbolic set $\Lam$ from \Cref{sec:part1}, we have a blender $(\Lam, \Fcal)$.

\subsubsection{Activation}
\label{sec:activation}

To apply \Cref{lemma:activate} we need to find a strong dynamical box $B_0$ containing the fixed point and a sequence of weak dynamical boxes connecting $B_0$ to the robust covering collection $\mathcal{F}$ by weak covering relations.

To check that $B_0$ contains the fixed point, we check that it strongly covers itself by applying the \texttt{strongCovering} test from \Cref{sec:part1} using $f^{-1}$ in place of $f$ and reordering the coordinate axes, reducing the problem to a $u = 2$ covering for the inverse.

Then we build, with the function \texttt{findActivationOrbit}, a sequence of boxes $B_1, \dots, B_l$, and we verify that this sequence satisfies the hypothesis of \Cref{lemma:activate} (with \texttt{confirmActivationOrbit}).

To get this sequence, we sample $1000$ points approximately in the local unstable manifold of $p_0 = (0, 0, 0)$, and for each one of these points we build a dynamical box $B_1$.

Since the system is partially hyperbolic, in local coordinates, the $x$ direction is always expanding, the $z$ direction is always contracting and the center direction given by $y$ can have any contracting or expanding rate in between. Thus, when we check for the weak covering relation with $u = 1$ and $s = 2$ (see \Cref{def:weakCov}) in the \texttt{curveCovering} function, described at the end of \Cref{sec:checkCov}, the center contraction may fail to hold.

To guarantee that $\pi_y(\floc(\bm{E})) \subset [-1,1]$ holds, where $\bm{E}$ is the subcube we use in the \texttt{curveCovering} function, we build $B_1$ with the function \texttt{fitBox}. This function multiplies the center direction of the dynamical box by a constant $c$ as small as possible such that the inequalities related to the center direction in \texttt{curveCovering} hold, and creates a cone as narrow as possible that contains the image of the cone of the first box.

A weak covering relation between the two cone boxes is then confirmed by calling \texttt{curveCovering} together with a cone inclusion test (checking that $D\floc \Cone_0 \subof \Cone_1$), see \Cref{rem:covRel-implies-weakCov}.

We then build a dynamical box $B_2$ around the image of the middle point of $B_1$. By iterating this process and stopping when the middle point of $B_l$ enters $V$ (or when we reach the maximum number of iterates set as $8$ in the function \texttt{orbitToV}), we have a list of dynamical boxes $B_0, \cdots, B_l$ satisfying
$$B_0 \fcovers B_0 \fweakcovers B_1 \fweakcovers \cdots \fweakcovers B_l,$$
where $l < 8$. The function \texttt{findActivationOrbit} returns this sequence if it activates the blender (by \Cref{lemma:activate}), which is confirmed by the function \texttt{confirmActivationOrbit} by checking that $B_0 \fcovers B_0$ and $B_0 \fweakcovers B_1 \fweakcovers \cdots \fweakcovers B_l$ as explained above, and that $B_l \fweakcovers \Fcal$ with the \texttt{activatesBlender} routine.

The function \texttt{activatesBlender} checks that the image of $B_l$ under $f$ crosses $V$ (the $z$-coordinate stays in $V_z$, the $x$-coordinate crosses $V_x + n_x$ for some integer $n_x$), and for every bunch $\Fcal_i$ whose rectangle intersects the image of $B_l$ at the $q_x$-plane, the derivative maps the cone of $B_l$ into the cone of $\Fcal_i$, checked with \texttt{coneInclusion}.

\subsection{Partial hyperbolicity}
\label{sec:part3}

We can apply \Cref{prop:1ph-proof} for the particular case where the atlas is composed by dynamical boxes $B_i$, each one equipped with constant cone fields, as stated in \Cref{prop:weakph}.

Partial hyperbolicity is verified as part of \texttt{verifyWeak} , which is called twice: once for the forward dynamics (\texttt{verifyUnstable}) and once for the backward dynamics (\texttt{verifyStable}), see also \Cref{sec:part4}.

We describe the test for weak partial hyperbolicity for the forward dynamics, that results in the splitting $E^{cs} \oplus E^u$, with the splitting $E^s \oplus E^{cu}$ being obtained analogously for the backward dynamics.

We first split the torus $\mathbb{T}^3$ into boxes $\bm{s_i}$ that we call \textit{seeds}, each one given by a triple of intervals, and then we cover each one of these seeds with dynamical boxes $B_i$ given by \texttt{buildSeedBoxes} (see \Cref{sec:coneBox}). Unlike in \Cref{sec:part1} where the ``slope'' $a$ for the cone is irrelevant, since we test for their invariance using the quadratic form, here the cones are the ones given by \Cref{eq:cones}, narrower and defined by the slope $a \in (0, 1]$. 

For all $B_i$, we consider the constant cone field in $\cube$ given by $\Cone^u(a)$ (see \Cref{eq:cones}). To apply \Cref{prop:weakph} we only need the cones in local coordinates, and for implementation it is convenient to have the same cones for every cone box.

For every $i, j \in \{0, \; \cdots, \;N-1\}$, where $N$ is the number of boxes covering $\mathbb{T}^3$, we check if the interval enclosure $[f(\bm{s_i})]$ intersects $\bm{s_j}$. If that fails to happen, then there is nothing to be checked. If $[ f(\bm{s_i}) ] \cap \bm{s_j} \neq \varnothing$, we check the following:
\begin{enumerate}    
    \item \emph{u-invariance}: $[D\floc(\cube)] (\Cone^u(a)) \subseteq \Cone^u(a)$;
    \item \emph{u-expansion}: $[\Vert [D\floc(\cube)] (\Cone^u(a)) \Vert] > [\Vert \Cone^u(a) \Vert]$,
\end{enumerate}
where $\floc = \varphi_j^{-1} \circ f \circ \varphi_i$.

We can still have ``false positives'' for the intersection test, but that only means that we perform more verifications of partial hyperbolicity than necessary.

The algorithm to test the conditions of \Cref{prop:weakph} includes the construction of the family $\{B_i\}_{i}$, and the verification of the conditions above in the function \texttt{verifyWeakPH}. 

Observe that $N = n^3$, where $n = 40$ for the forward map $f$ or $n = 60$ for the inverse $f^{-1}$, and there are $n^6$ pairs to be tested, since we have to verify if every pair of boxes intersects. If they do, we have to check expansion and invariance. Thus, this is the most time-consuming step of the program. Some pairs can be avoided by using the fact that the seeds are constructed by the function \texttt{buildSeedBoxes} by looping over the $x$-coordinate first, and then over the $y$-coordinate. So, fixed $i$, if image $f(\bm{s_i})$ does not intersect the $x$-coordinate of the first box with that interval, we can skip $n^2 - 1$ verifications with all $\bm{s_j}$ with that interval in the first coordinate. We do the same for the $y$-coordinate. See the function \texttt{verifyWeakPH} in the code for the details.

\subsection{Dense stable and unstable sets}
\label{sec:part4}

We explain the implementation of \Cref{prop:densestable} for $f$ in \texttt{verifyUnstable}, with the implementation for $f\inv$ in \texttt{verifyStable} being analogous, both using the \texttt{verifyWeak} function.

We first build two dynamical boxes around the fixed point $p_0 = (0, 0 ,0)$, first a \textit{huge box} $H_0$ that serves as a target for all points in the manifold, and a \textit{tiny box} $T_0$ that is used to confirm the hyperbolicity of the fixed point $p_0$.

We use the function \texttt{shrink} to find a sequence of weak dynamical boxes $X_1, \cdots, X_{m}$ satisfying
$$H_0 = X_1 \fweakcovers X_2 \fweakcovers \cdots \fweakcovers X_m = T_0.$$

The \texttt{hyperbolicFixedPoint} test then confirms that $T_0 \fcovers T_0$, so the tiny box contains a unique hyperbolic fixed point with unstable index equals $1$.

For the backward dynamics (\texttt{verifyStable}), the code additionally checks that the tiny box at $q_0$ is contained in the blender region $V_x \ti S^1 \ti V_z$, confirming $q_0 \in \Lambda$.

Then, to apply \Cref{prop:densestable}, it only remains to cover $M$ with seed boxes that are good seeds for dynamical boxes $\{B_i\}$ satisfying $B_i \fweakcovers H_0$ for all $i$. 

We use the same boxes and cones constructed in \Cref{sec:part3} to check weak partial hyperbolicity. We check that each $\bm{s_i}$ is a good seed for $B_i$ with the function \texttt{allCurvesAreGood}. It verifies that every $C^1$ curve that passes through the seed and is tangent to the cone exits the cone box through the left and right faces (the u-boundary), see the definition of good seed in \Cref{sec:activation-crit}.

Finally, we check that $B_i \fweakcovers H_0$ for all $i$ with the routine \texttt{allBoxesCoverTarget}. This implies that $\hypW^s(p_0)$ is dense in $\bbT^3$.

\appendix \section{Symbolic proof of transitivity}
\label{appa}

We give a proof of the symbolic coding given by \Cref{thm:coding}, which in particular allows us to prove the existence of a transitive invariant set in \Cref{prop:lam}.

\begin{proof}[Proof of \Cref{thm:coding}]
    Given a sequence $\{a_n\} \in \Sig_T$, we construct a unique point $x \in M$ satisfying $f^n(x) \in B_{a_n}$ for all $n \in \bbZ$.

    A standard graph transform argument gives us that there is a sequence $\tilde{\rho}_n \subseteq B_{a_n}$ of u-plaques such that $\tilde{\rho}_{n+1} = f(\tilde{\rho}_n) \cap B_{a_{n+1}}$. We include its proof here.
    
    Consider the set
    $$\mathcal{L} = \left\{\rho_{\loc}: \bbZ \times [-1, 1]^u \to [-1, 1]^s \, : \, \begin{aligned}
        & \mbox{ for every } n \in \bbZ, \; \rho_{\loc}(n, \cdot) \mbox{ is }\\
        & \mbox{ Lipschitz with uniform} \\
        & \mbox{ Lipschitz constant $C > 0$} ; 
    \end{aligned}\right\}$$ 
    of sequences of Lipschitz plaques in $\cube$ with uniform Lipschitz constant $C > 0$, where $\cube = [-1, 1]^{u+s}$. In $\mathcal{L}$, consider the metric
    $$\mathcal{D}(\rho, \gam) = \sup\{\vert\rho(n, x) - \gam(n, x) \vert \, : \, n \in \bbZ, \, x \in [-1,1]^u\},$$
    that is easily verified to make $(\mathcal{L}, \mathcal{D})$ a complete metric space. We would like to define a graph transform on $\mathcal{L}$ induced by the action of $f$ on each graph, but that is not well-defined \textit{a priori}. Consider, then, its subspace
    $$\mathcal{L}_1 = \left\{\rho_{\loc}: \bbZ \times [-1, 1]^u \to [-1, 1]^s \, : \, \begin{aligned}
        &\rho_n = \varphi_{a_n}(\graph(\rho_{\loc}(n,\cdot))) \\
        &\mbox{ is a u-plaque in } B_{a_n}
    \end{aligned}\right\}$$  
    of sequences defining u-plaques in local coordinates.

    Consider now the function $\mathcal{F}: \mathcal{L}_1 \to \mathcal{L}_1$ defined by $\mathcal{F}(\rho) = \hat{\rho}$ where
    $$\graph(\hat{\rho}(n+1, \cdot)) = f_n\left(\graph(\rho(n, \cdot))\right) \cap R,$$
    where $f_n = \varphi_{a_{n+1}}^{-1} \circ f \circ \varphi_{a_n}$ is the local map. This is well-defined by the definition of strong covering, which guarantees that $f_n\left(\graph(\rho(n, \cdot))\right) \cap \cube$ has only one connected component (since there is a unique subplaque in the interior of $\rho_n$ such that its image is a subplaque in $B_{a_n}$) that is exactly a u-plaque. Also, it follows directly from the definition of strong dynamical box that $F$ is a contraction, since $f_n$ expands uniformly along u-plaques and contracts uniformly along s-plaques, and they are transversal.

    Since $(\mathcal{L}, \mathcal{D})$ is complete, the closure $\overline{\mathcal{L}_1}$ is a complete subspace, so by the Banach fixed-point theorem, $F$ has a fixed point $\tilde{\rho} \in \overline{\mathcal{L}_1}$. This implies that the graph of $\tilde{\rho}(n, \cdot)$ is mapped under $f_n$ in $\cube$ to the graph of $\tilde{\rho}(n+1, \cdot)$. One can prove that each $\tilde{\rho}(n, \cdot)$ is $C^1$ (it is a standard argument appearing in the proof of the Hadamard--Perron Theorem, see for instance Step 5 in \cite[Theorem 6.2.8]{KH1995}), and since the fixed point is the limit of the iteration of any element of $\mathcal{L}_1$, the cone preservation of item (1) in \Cref{def:strongCov} makes each $\tilde{\rho_n} = \varphi_{a_n}(\tilde{\rho}(n,\cdot))$ a u-plaque. Thus we have the sequence of u-plaques as claimed. 
    
    Now we show that there is a unique sequence of points $x_n \in \cube$ satisfying $f_n(x_n) = x_{n+1}$ for all $n \in \bbZ$, which implies the existence of the desired $x \in M$ as $x = \varphi_{a_0}(x_0)$.

    Since $f_n^{-1}(\tilde{\rho}(n+1, \cdot))$ is a unique connected subplaque $\alpha_n$ in $\cube$ and $f$ is a diffeomorphism, $\restr{f_n}{\alpha_n}$ is a bijection for all $n \in \bbZ$. So for all $p \in \graph(\tilde{\rho}(n+1, \cdot))$ there is a unique $q \in \graph(\tilde{\rho}(n, \cdot))$ such that $f_n(q) = p$.
    
    If $\mathcal{S}$ is the space of sequences in $\cube$ such that $p_n \in \graph(\tilde{\rho}(n, \cdot))$, consider the metric $\mathcal{D}_{\rho}$ in $\mathcal{S}$ given by
    $$\mathcal{D}_{\rho}(\{p_n\}, \{r_n\}) = \sup_n d_{n}(p_n, r_n),$$
    where $d_n$ is the distance along $\graph(\tilde{\rho}(n, \cdot))$.

    The map $\mathcal{G}: \mathcal{S} \to \mathcal{S}$ that takes $\{p_n\}$ to $\{q_n\}$, given by $q_n = f_n \inv(p_{n+1})$, is well-defined and it is a contraction. So, since $(\mathcal{S}, \mathcal{D}_{\rho})$ is easily verified to be complete, $\mathcal{G}$ has a unique fixed point $\{x_n\}$ that satisfies
    \begin{equation}
        \label{eq:point-invariant}
        f_n(x_n) = x_{n+1} \mbox{ for all } n \in \bbZ.
    \end{equation}
    It remains to verify that $\{x_n\}$ is the only sequence of points in $\cube$, not necessarily a sequence in $\mathcal{S}$, that satisfies \Cref{eq:point-invariant}. If $\{p_n\}$ is any sequence of points in $\cube$, suppose that $\{p_n\}$ satisfies \Cref{eq:point-invariant}. Denoting $p_n = (p^u_n, p^s_n)$, we can consider $\mathcal{L}_2 \subseteq \mathcal{L}_1$ given by  
    $$\mathcal{L}_2 = \left\{\rho \in \mathcal{L}_1: \rho(n, p^u_n) = p^s_n \right\}.$$
    We have that $\mathcal{L}_2$ is $\mathcal{D}$-closed, non-empty and $\mathcal{F}$-invariant. So the fixed point $\tilde{\rho}$ of $\mathcal{F}$ should be in $\mathcal{L}_2$. Thus, $\{p_n\}$ is in $\mathcal{S}$ and, by the above discussion about $\mathcal{G}$, it is unique.

    Hence the map $h : \Sig_T \to M$, $\{a_n\} \mapsto x$, is well-defined, and $h \circ \sig = f \circ h$ follows immediately from the definition. The image $h(\Sig_T)$ is uniformly hyperbolic because the strong covering relations together over finitely many boxes provide a uniformly expanding and $Df$-invariant cone field $\Cone^u$ and a uniformly contracting and $Df^{-1}$-invariant cone field $\Cone^s$ on $h(\Sig_T)$.

    Finally, the characterizations in the statement hold as
    $$\tilde{\rho}_0 = \mbox{ connected component through $x$ of } \hypW^u(x) \cap B_{a_0}.$$
    Indeed, if $p \in \varphi_{a_0}(\graph(\tilde{\rho}(0, \cdot)))$, then the fact that $\tilde{\rho}$ is a fixed point for $\mathcal{F}$ and each $\tilde{\rho}_n$ is a u-plaque implies that, for all $n \in \bbN$, the set $\bt_n = f^{-n}(\tilde{\rho}_0) \subseteq \tilde{\rho}_{-n}$ is such that $f^{-n}(p) \in \bt_n$, and that the diameter of $\bt_n$ converges to $0$ as $n$ goes to $\infty$. But $f^{-n}(x) \in \bt_n$ by the construction of $x$, so both $f^{-n}(x)$ and $f^{-n}(p)$ belong to $\bt_n$, which implies that $p$ belongs to $\hypW^u(x)$.

    For the reciprocal inclusion, take $p$ in the connected component through $x$ of $\hypW^u(x) \cap B_{a_0}$. Consider, for all $n \in \bbN$, $p_n$ as the $n^{\text{th}}$ pre-image of $p$ in local coordinates: $p_n = \varphi^{-1}_{a_n}(f^{-n}(p)) = (p^u_n, p^s_n)$. Consider an element $\rho$ of $\mathcal{L}_1$ such that, for all $n \in \bbN$, it satisfies $\rho(-n,\cdot) \equiv p^s_n$, meaning that for all negative indices it is a flat u-plaque passing through $p_n$. Since $\tilde{\rho}$ is a fixed point for $\mathcal{F}$, we have that $D(\mathcal{F}^k(\rho), \tilde{\rho})$ converges to $0$ as $k$ goes to $\infty$. But $\varphi^{-1}_{a_0}(p) \in \graph(\mathcal{F}^k(\rho(-k, \cdot)))$ by construction, so $\varphi^{-1}_{a_0}(p) \in \graph(\tilde{\rho}(0, \cdot))$, as desired. The characterization for the stable manifold is analogous.
\end{proof}

\bibliographystyle{amsalpha}
\bibliography{references}
\end{document}